\documentclass[11pt]{amsart}
\usepackage{amssymb, latexsym, amsmath, amsfonts}

\newtheorem{thm}{Theorem}[section]
\newtheorem{cor}[thm]{Corollary}
\newtheorem{lem}[thm]{Lemma}
\newtheorem{prop}[thm]{Proposition}
\theoremstyle{definition}

\theoremstyle{remark}

\numberwithin{equation}{section}

\newcommand{\mbb}{\mathbb}
\newcommand{\ra}{\rightarrow}
\newcommand{\z}{\zeta}
\newcommand{\pa}{\partial}
\newcommand{\ov}{\overline}
\newcommand{\sm}{\setminus}
\newcommand{\ep}{\epsilon}
\newcommand{\no}{\noindent}
\newcommand{\Om}{\Omega}
\newcommand{\cal}{\mathcal}
\newcommand{\ti}{\tilde}
\newcommand{\la}{\lambda}

\newcommand{\tim}{\times}

\begin{document}
\title{Holomorphic mappings between domains in $\mbb C^2$}
\thanks{KV was supported in part by a grant from UGC under DSA-SAP, Phase IV; R.S. was supported 
in part by the Natural Sciences and Engineering Research Council of Canada}
\author{Rasul Shafikov, Kaushal Verma}
\address{Department of Mathematics, The University of Western Ontario,
London, Ontario N6A 5B7, Canada}
\email{shafikov@uwo.ca}
\address{Department of Mathematics,
Indian Institute of Science, Bangalore 560 012, India}
\email{kverma@math.iisc.ernet.in}


\maketitle

\setcounter{tocdepth}{1}
\tableofcontents

\section{Introduction}

\no The question of determining the boundary behaviour of a biholomorphic or more generally a proper 
holomorphic mapping between two given domains in $\mathbb C^n$ ($n > 1$) is well known. In particular, 
if $f : D \ra D'$ is a proper holomorphic mapping between smoothly bounded domains in $\mathbb C^n$, 
the conjecture that $f$ extends smoothly up to $\pa D$, the boundary of $D$, remains open in complete 
generality. We shall henceforth restrict ourselves to the case when the boundaries are smooth real analytic, 
where much progress has recently been made. The main result of this article is:

\begin{thm}
Let $D, D'$ be domains in $\mathbb C^2$, both possibly unbounded, and $f : D \ra D'$ a holomorphic mapping. 
Let $M \subset \pa D$ and $M'\subset \pa D'$ be open pieces, which are smooth real analytic and of finite type, 
and fix $p \in M$. Suppose there is a neighbourhood $U$ of $p$ in $\mathbb C^2$ such that the cluster set of 
$U \cap M$ does not intersect $D'$. Then $f$ extends holomorphically across $p$ if one of the following conditions 
holds:

\begin{enumerate}
\item[(i)] $p$ is a strongly pseudoconvex point, and the cluster set of $p$ contains a point in $M'$.
\item[(ii)] $M$ is pseudoconvex near $p$, and the cluster set of $p$ is bounded and contained in $M'$.
\item[(iii)] $D$ is bounded, $f : D \ra D'$ is proper, and the cluster set of $M$ is contained in $M'$.
\end{enumerate}
\end{thm}

\no The following examples that have been borrowed from \cite{DPJGA} and \cite{PT} show that holomorphic extendability of $f$ cannot be hoped for in the absence of the 
hypotheses considered above.

\medskip

\no {\it Example 1:} Let $D = \{z \in \mathbb C^2 : 2 \Re z_2 + \vert z_1 \vert^2 < 0\}$ and 
$D' = \{z \in \mathbb C^2 : 2 \Re (z'_2)^2 + \vert z'_1 \vert^2 < 0\}$. Note that $0 \in \pa D$ and 
$0' \in \pa D'$. Furthermore, $D \backsimeq \mathbb B^2$, the unit ball in $\mathbb C^2$, while $D'$ 
has two connected components none of which has smooth boundary near the origin. However $\pa D'$ is 
a real analytic set. Since $z_2 \not= 0$ in $D$, it is possible to choose a well defined branch of $\sqrt{z_2}$ 
in $D$, and having made such a choice let $f(z_1, z_2) = (z_1, \sqrt{z_2})$. Then $f$ is a biholomorphism 
between $D$ and a connected component say $D'_1$ of $D'$. Moreover $0' \in cl_f(0)$. But $f$ does not 
extend holomorphically across $0 \in \pa D$.

\medskip

\no {\it Example 2:} Let $D = \{z \in \mathbb C^2 : 2 \Re z_2 + \vert z_1 \vert^2 < 0\}$ and $D' = \{z \in \mathbb C^2 : 2 \Re z'_2 + \vert z'_1 \vert^2 
\vert z'_2 \vert^2 < 0\}$ and $f(z_1, z_2) = (z_1/z_2, z_2)$. Again $0 \in \pa D, 0' \in \pa D'$ and $f : D \ra D'$ is a biholomorphism. Both $\pa D, 
\pa D'$ are smooth real analytic but the curve $\z \mapsto (\z, 0), \z \in \mathbb C$, is contained in $\pa D'$. Consequently $\pa D'$ is not of finite 
type near $0'$. Choose a sequence of positive reals $\ep_j \ra 0$ and observe that $f(0, -\ep_j) = (0, -\ep_j)$. This shows that $0' \in cl_f(0)$ but 
$f$ evidently does not have a holomorphic extension across $0 \in \pa D$. Other examples of biholomorphisms from $D$ that are of a similar nature can be 
constructed by considering rational functions on $\mathbb C^2$ whose indeterminacy locus contains the origin.

\medskip

\no {\it Example 3:} Let $g(z)$ be an inner function on $\mathbb B^n$ ($n > 1$). Then $g(z)$ has unimodular boundary values almost everywhere on $\pa \mathbb 
B^n$ by Fatou's theorem. However it is known (see Proposition~19.1.3 in \cite{Ru} for example) that there is a dense $G_{\delta}$-subset of $\pa \mathbb
B^n$ such that the cluster set of any point from it must intersect the unit disc $\Delta \subset \mathbb C$. In fact the image of any radius in 
$\mathbb B^n$ that ends at a point on this $G_{\delta}$ under $g(z)$ is dense in $\Delta$. Therefore, $f(z) = (g(z), 0, \ldots, 0)$ is a 
holomorphic self map of $\mathbb B^n$ for which there does not exist a neighbourhood $U$ of any given boundary point with the property that the cluster set of $U \cap \pa 
\mathbb B^n$ does not intersect $\mathbb B^n$. Evidently $f$ fails to admit even a continuous extension to any $p \in \mathbb B^n$.

\medskip

\no To provide a context for this theorem, recall the theorems of Diederich-Pinchuk from \cite{DP1} and \cite{DP2}. In \cite{DP1} it was shown that  
every proper holomorphic mapping between smoothly bounded real analytic domains in $\mathbb C^2$ extends holomorphically across the boundaries while 
\cite{DP2} contains a similar statement for continuous CR mappings (that are a priori non-proper) between smooth real analytic finite type hypersurfaces 
in $\mathbb C^n$, $n > 2$. The theorem above shows that it is possible to study the boundary behaviour of $f$ under purely local hypotheses. Motivated in part by 
\cite{BerSu}, \cite{Su}, and \cite{ForLow}, which deal with similar local theorems under convexity assumptions on the boundaries
either of a geometric or a function theoretic nature, attempts to arrive at such a local statement were made in \cite{V} and \cite{SV}, both of which were proved
under additional hypotheses only on the mappings involved. It should be noted that local extension theorems for proper holomorphic mappings across pseudoconvex
real analytic boundaries played a particularly useful role in the global theorem of \cite{DP1}. Cases (i) and (ii) provide an instance of such extension theorems
even without properness of $f$. Case (i) in particular shows that $f$ can be completely localized near a strongly pseudoconvex point as soon as its cluster set contains 
a smooth real analytic finite type point. Thus, these statements cover the case of infinite sheeted coverings $f : D \ra D'$. Second, in (ii) no assumptions are being 
made about the cluster set of points on $M$ close to $p$. In particular, the possibility that the cluster set of a point $q \in M$ close to $p$ contains the point at infinity 
in $\pa D'$ is apriori allowed -- that this cannot happen will follow from the boundedness of the cluster set of $p$ and will be explained later. Finally, a word about the 
various hypotheses being made in (i), (ii) and (iii) about the relative location of $p \in M$ -- while pseudoconvex points are taken care of by the first two cases, 
the content of (iii) lies in the fact that it addresses the remaining possibility that $p$ is either on the border between the pseudoconvex and pseudoconcave points on 
$M$, or $M$ is locally pseudoconcave near $p$. However, it is well known that $f$ extends across pseudoconcave points. Therefore, it is sufficient to consider the case when 
$p$ is on the border between the pseudoconvex and pseudoconcave points on $M$. The proof of these theorems falls within the purview of the geometric reflection principle 
as developed in \cite{DW}, \cite{DF}, \cite{DP1} and \cite{DP2}. Further refinements of these techniques from \cite{SV} and the ideas of analytic continuation of germs of 
holomorphic mappings along paths on smooth real analytic hypersurfaces from \cite{Sh} are particularly useful.

\medskip

The geometric reflection principle as developed in \cite{DW}, \cite{DF}, \cite{DP1} studies the influence of a holomorphic mapping between two 
smoothly bounded real analytic domains on the Segre varieties that are associated to each boundary point. A quick review of what will be needed later 
is given below but detailed proofs that can be found in the aforementioned references have been skipped, the purpose of the exposition being solely 
to fix notation. For brevity we work in $\mathbb C^2$ as the case for $n > 2$ is no different and write $z = (z_1, z_2)$ for a point $z \in \mathbb C^2$. 
Let $\Om \subset \mathbb C^2$ be open and $M \subset \Om$ a closed, smooth real analytic real hypersurface. Pick $\z \in M$ and translate coordinates so 
that $\z = 0$. Let $r(z, \ov z)$ be the defining function of $M$ in a neighbourhood, say $U$ of the origin and suppose that $\pa r / \pa z_2 (0) 
\not= 0$. Let $U^{\pm} = \{z \in U: \pm r(z, \ov z) > 0\}$. If $U$ is small enough, the complexification $r(z, \ov w)$ of $r(z, \ov z)$ is well 
defined by means of a convergent power series in $U \times U$. It may be noted that $r(z, \ov w)$ is holomorphic in $z$ and anti-holomorphic in $w$. For 
any $w \in U$, the associated Segre variety is
\[
Q_w = \{z \in U : r(z, \ov w) = 0\}.
\]
By the implicit function theorem $Q_w$ can be written as a graph for each $w \in U$. In fact, it is possible to choose a pair of neighbourhoods 
$U_1, U_2$ of the origin with $U_1$ compactly contained in $U_2$ such that for any $w  \in U_1$, $Q_w \subset U_2$ is a closed complex 
hypersurface and 
\[
Q_w = \{z = (z_1, z_2) \in U_2 : z_2 = h(z_1, \ov w)\},
\]
where $h(z_1, \ov w)$ depends holomorphically on $z$ and anti-holomorphically on $w$. Such neighbourhoods are usually called a {\it standard pair} of 
neighbourhoods and for convenience they will be chosen to be polydiscs around the origin. $Q_w$ as a set is independent of the choice of $r(z, \ov 
z)$ because any two local defining functions for $M$ near the origin differ by a non-vanishing multiplicative factor, and this persists upon 
complexification as well. A similar argument shows that Segre varieties are invariantly attached to $M$ in the following sense: let $M'$ be another smooth 
real analytic hypersurface in $\mathbb C^2$. Pick $p, p'$ on $M, M'$ respectively and open neighbourhoods $U_p, U'_{p'}$ containing them. If $f : U_p \ra 
U'_{p'}$ is a holomorphic mapping such that $f(U_p \cap M) \subset U'_{p'} \cap M'$ then for all $w \in U_p$, $f(Q_w) \subset Q'_{f(w)}$ where $Q'_{f(w)}$ 
denotes the Segre variety associated to $f(w) \in U'_{p'}$. This will be referred to as the invariance property of Segre varieties. It forms the basis 
of constructing the complex analytic set that is used to extend holomorphic mappings across real analytic boundaries. The practice of distinguishing 
analogous objects in the target space by adding a prime will be followed in the sequel. For $\z \in Q_w$ the germ of $Q_w$ at $\z$ will be denoted by 
${}_{\z}Q_w$. Let $\cal S = \{Q_w : w \in U_1\}$ be the ensemble of all Segre varieties and let $\la : w \mapsto Q_w$ be the so-called Segre map. Then 
$\cal S$ admits the structure of a complex analytic set in a finite dimensional complex manifold. The fibres
\[
I_w = \la^{-1}(\la (w)) = \{z : Q_z = Q_w\}
\]
are then analytic subsets of $U_1$ and for $w \in M$, it can be shown that $I_w \subset M$. If $M$ does not contain germs of positive dimensional 
complex analytic sets, i.e., it is of finite type in the sense of D'Angelo, it follows that $I_w$ is a finite collection of points. Other notions of finite 
type, such as in the sense of Bloom-Graham or essential finiteness, are discussed in more detail in \cite{BER}; in $\mathbb C^2$ however all these notions are 
equivalent. Assuming now that $M$ is of finite type it follows that $\la$ is proper in a neighbourhood of each point on $M$. Also note that $z \in Q_w$ is 
equivalent to $w \in Q_z$ and that $z \in Q_z$ iff $z \in M$, both of which are consequences of the reality of $r(z, \ov z)$. The notion of a symmetric 
point from \cite{DP1} will be useful here as well: for a fixed $w \in U_1$, the complex line $l_w$ containing the real line through 
$w$ and orthogonal to $M$ intersects $Q_w$ at a unique point. This is the symmetric point of $w$ and will be denoted by ${}^sw$. If $w \in M$ then $w = 
{}^sw$ and it can be checked that for $w \in U^{\pm}$, the symmetric point ${}^sw \in U^{\mp}$. Finally if $w \in U^+$, the component of $Q_w \cap U^-$ that contains ${}^sw$ will be denoted by $Q^c_w$ and referred to as the canonical component.


\section{Remarks on the proof of the theorem}

\no By shrinking the neighbourhood $U$ of $p$ whose existence is assumed in the theorem, we may suppose  
that $U \cap M$ is a smooth real analytic hypersurface of finite type. The standard pair of neighbourhoods 
$U_1 \subset U_2$ of $p$, needed to define the Segre varieties associated to points on $U \cap \pa D$ near $p$, 
will then be chosen to be compactly contained in $U$. Let $p'$ be a point on $M'$ that lies in $cl_f(p)$,
the cluster set of $p$. Likewise, fix a neighbourhood $U'$ of $p'$ so that $U' \cap \pa M'$ is again a smooth real 
analytic hypersurface of finite type and then fix a standard pair $U'_1 \subset U'_2$ around $p'$. Abusing notation, 
we shall denote $U \cap M$ and $U' \cap \pa M'$ by $M, M'$ respectively. The following stratification of $M$ from \cite{DP1} will be needed: 
let $T$ be the set of points on $M$ where its Levi form vanishes. Then $T$ admits a semianalytic stratification as $T = T_0 \cup T_1 \cup T_2$ where $T_k$ 
is a locally finite union of smooth real analytic submanifolds of dimension $k = 0, 1, 2$ respectively. Denote by $M^{\pm}_s$ the set of strongly pseudoconvex 
(resp. strongly pseudoconcave) points on $M$. Let $M^{\pm}$ be the relative interior, taken with respect to the relative topology on $M$, of the closure of $M^{\pm}_s$. 
Then $M^{\pm}$ is the set of weakly pseudoconvex (resp. weakly pseudoconcave) points of $M$ and the border $M \sm (M^+ \cup M^-) \subset T$ separates $M^+$ and $M^-$. 
It was shown in \cite{DFY} and \cite{DP1} that this stratification of $T$ can be refined in such a way that the two dimensional strata become maximally totally real manifolds. 
Retaining the same notation $T_k, k = 0, 1, 2$ for the various strata in the refined stratification, let $T^+_k = M^+ \cap T_k$ for all $k$. Then $T^+_2$ 
is the maximally totally real strata near which $M$ is weakly pseudoconvex. It was shown in \cite{DFor} that
\[
(M \sm (M^+ \cup M^-)) \cap T_2 \subset M \cap \hat D
\]
where $\hat D$ is the holomorphic hull of the domain $D$. Evidently $f$ holomorphically extends to a neighbourhood of each point on $(M \sm (M^+ \cup 
M^-)) \cap T_2$. Its complement in $T$ is 
\[
M_e = (M \sm (M^+ \cup M^-)) \cap (T_1 \cup T_0)
\]
which will be called the exceptional set. Observe that each of $M_e$ and $T^+_1 \cup T^+_0$ is a locally finite union of real analytic arcs and 
points. The point $p$ can then lie in either $M^+_s, T^+_2 \cup T^+_1 \cup T^+_0, M_e$ or $\big( M \sm (M^+ \cup M^-) \big) \cap T_2$ and the same 
possibilities hold for $p'$ by considering the corresponding strata on $M'$.

\medskip

The first thing to prove in cases (i) and (ii) of the theorem is that the restriction $f : U^- \ra D'$ has discrete fibres. This will guarantee that the set defined by
\begin{equation}\label{eq:X}
X_f = \big\{ (w, w') \in U^+_1 \tim U'^+_1 : f(Q_w \cap D) \supset {}_{{}^sw'}Q'_{w'} \big\},
\end{equation}
if non-empty, is a locally complex analytic set. In the setting of cases (i) and (ii), there are negative plurisubharmonic barriers at boundary points near $p$
(see the next section for details) and thus by \cite{Su}, $f$ is known to have discrete fibres. 
In Proposition 3.2 we show this near any smooth finite type boundary point.
Let $\pi, \pi'$ denote the projections from $U \tim U'$ onto the first and second factor 
respectively. Since $\la$, $\la'$ are proper near $p$, $p'$ respectively, the fact that $f$ has discrete fibres near $p$ forces both projections $\pi : X_f \ra 
U^+_1$ and $\pi' : X_f \ra U'^+_1$ to have discrete fibres as well. Now no assumptions on the global cluster set of $p$ are being made and hence it is not 
possible to directly prove that $X_f$ is contained in a closed analytic subset of $U_1 \tim U'_1$ with proper projection onto $U_1$. If this were 
possible, $f$ would then extend as a holomorphic correspondence and hence as a mapping by Theorem 7.4 of \cite{DP1}. 

\medskip

To illustrate the salient features of the proof of case (i) note that $p \in M^+_s$ and that $p' \in M'^+, T'$ or $M'^-$. 
The case when $p' \in M'^+$ is well understood for 
there are local plurisubharmonic peak functions near $p, p'$ which lead to the continuity of $f$ near $p$. When $p' \in M'^-_s$, the local plurisubharmonic peak 
function near $p$ (which can be continuously extended to all of $D$) can be pushed forward by $f$ to get a negative plurisubharmonic function on $D'$. The strong 
pseudoconcavity of $M'$ near $p'$ implies the existence of complex discs in $D'$ near $p'$ whose boundaries are uniformly compactly contained in $D'$. The restriction 
of this negative plurisubharmonic function on $D'$ to these discs is shown to violate the maximum principle. When $p' \in M'^-$ the graph of $f$ is shown to extend to 
a neighbourhood of $(p, p')$ as an analytic set and hence as a holomorphic mapping by \cite{DPJGA}, and this leads to a contradiction. The case when $p' \in \big( M' \sm 
(M'^+ \cup M'^-) \big) \cap T'_2$ requires several intermediate steps -- indeed, as noted above, such a $p' \in \hat D'$ and if $f : D \ra D'$ was proper then all points in 
the cluster set of $p'$ under the correspondence $f^{-1}$ on $\pa D$ would belong to $\hat D$ (see Lemma 3.1 in \cite{DP1}), and hence $f$ would extend holomorphically past 
$p$. This reasoning does not apply here for $f^{-1}$ is not known to be well defined even as a correspondence.
Therefore, we first show that there is a sequence $p^j \in U \cap M$ such that $f$ extends holomorphically 
across $p^j$ and $f(p^j) \ra p'$. For each $j$, consider the maximal possible extension of $f$ along $Q_{p^j}$ (cf. see (4.4)) which associates to each $p^j$ a one 
dimensional analytic set $C_j$ in a natural way. Secondly, it is shown that the cluster set of $\{ C_j \}$ (which is defined in the next section) contains points near 
which $X_f$ is defined. In particular, there are points on $Q_p$ over which $X_f$ is a local ramified cover. The proof of the extendability of $f$ is then completed 
by using ideas of analytic continuation of germs of holomorphic mappings along real hypersurfaces from \cite{Sh}. Similar ideas apply when $p' \in \big( M' \sm (M'^+ 
\cup M'^-) \big) \cap (T'_1 \cup T'_0)$ and the conclusion is that $f$ holomorphically extends across $p$ even when $p' \in T'$. This evidently leads to a 
contradiction since there are points near $p, p'$ at which $f$ is locally biholomorphic but the Levi form is not preserved. Thus if $p \in M^+_s$, then $p'$ is forced 
to be in $M'^+$ and consequently $f$ extends across $p$. In fact by \cite{CPS}, $p'$ must be a strongly pseudoconvex point on $M'$. 

\medskip

\no For case (ii), note that $p \in M^+, T$ or $M^-$. When $p \in M^+_s$ case (i) shows that $f$ extends past $p$. Note that points on $M^-$ and $(M \sm (M^+ \cup 
M^-)) \cap T_2$ belong to $\hat D$ and $f$ extends past $p$ in these cases as well. What remains to consider are the cases when $p \in T^+_2$ or $(M\sm (M^+ \cup M^-)) \cap (T_1 
\cup T_0)$, i.e., points on the one and zero dimensional strata of the border between the pseudoconvex and pseudoconcave points. Suppose that $p \in T^+_2$ and let
\[
\cal L = \bigcup_{w \in T^+_2} Q_w
\]
where $w$ is allowed to vary near $p$ on $T^+_2$. Let $U' \subset \mbb C^2$ be an open neighbourhood that contains $cl_f(p)$ such that $U' \cap M'$ is a closed smooth real 
analytic hypersurface of finite type. Then $X_f \subset U^+_1 \times U'^+$ is well defined. The goal will be to show that $X_f$ extends as an analytic set in $U_1 \times 
U'$ and the obstructions in doing this arise from the limit points of $X_f$ on $U^+_1 \times (U'\cap M')$. The limit points are shown to lie on $\cal L \times (U'\cap M')$. 
The extendability of $X_f$ will follow by first showing that $\ov X_f$ is analytic near $\cal L \times T'^-_2$ while the other limit points on $\cal L \times ((M \sm (M^+ 
\cup M^-)) \cap (T_1 \cup T_0))$ (which is pluripolar) are removable by Bishop's theorem. Thus the graph of $f$ will be contained in an analytic set defined near $\{p\} 
\times (U'\cap M')$ and this will imply that $f$ holomorphically extends past $p$. 

\medskip

A set of similar ideas work when $p \in (M \sm (M^+ \cup M^-)) \cap (T_1 \cup T_0)$ and can 
be applied to prove case (iii). An essential ingredient in this is Lemma 6.1 according to which the 
cluster set of such a $p$ cannot contain strongly pseudoconvex points on $M'$ -- and this is shown 
to hold even without $f$ being proper. This exhausts all possibilities for $p$ -- indeed, as noted above, 
if $p$ is on a two dimensional stratum on the border, $p \in \hat D$ and hence $f$ extends 
holomorphically across $p$.


\section{The fibres of $f$ near $p$ are discrete}

\no In Section 6 of \cite{DP2}, Diederich-Pinchuk pose a conjecture on the cluster set of a sequence of pure dimensional analytic sets all of which are defined 
in a fixed neighbourhood of a real analytic CR manifold in $\mathbb C^n$ of finite type in the sense of D'Angelo. To recall the set up, let $W \subset \mathbb C^n$ 
be open and $N \subset W$ a relatively closed, real analytic CR manifold of finite type. Suppose that $A_j \subset W$ is a sequence of closed analytic sets of 
pure fixed dimension $k \ge 1$. Define the cluster set of $\{A_j\}$ as
\[
{\rm cl}(A_j) = \big\{ z \in W : {\rm there \; is \; a \; sequence} \; z_j \in A_j \; {\rm such \; that} \; z \; {\rm is \; a \; limit \; point \; of} 
\; (z_j) \big\}.
\]
Their conjecture is that ${\rm cl}(A_j)$ is not entirely contained in $N$. Although open in general, proofs of the validity of this conjecture were given by 
them in \cite{DP2} when $n = 2$ and $N$ is a smooth real analytic hypersurface of finite type and in other instance (cf. Propositions~8.2 and~8.3 in 
\cite{DP2}) when more information is apriori assumed either about $k$ or about the structure of ${\rm cl}(A_j)$. Their proof of this conjecture in case $n = 2$ 
depends on the existence of suitable plurisubharmonic peak functions at the pseudoconvex points and along the smooth totally real strata of the Levi 
degenerate points on $N$. A different argument for the case of the totally real strata, based on the well known fact that an analytic set of positive 
dimension cannot approach such a submanifold tangentially can be given as follows.

\begin{prop}
Let $M \subset U \subset \mathbb C^2$ be a closed, smooth real analytic hypersurface of finite type and suppose that $A_j \subset U$ is a sequence of closed 
analytic sets of pure dimension $1$. Then ${\rm cl}(A_j)$ cannot be entirely contained in $M$.
\end{prop}

\begin{proof}
Suppose that ${\rm cl}(A_j) \subset M$. If possible, pick $a \in {\rm cl}(A_j) \cap M^+_s$. Let $\ti U$ be an open neighbourhood of $a$ chosen so small that $\ti 
U \cap M \subset M^+_s$ and such that there is a continuous plurisubharmonic function $\phi$ on $\ti U$ with $\phi(z) < \phi(a)$ for all 
$z \in (\ti U \cap M) \sm \{a\}$. Since $a \in {\rm cl}(A_j)$ it follows that $A_j \cap \ti U \not= \emptyset$ for infinitely many $j$, but note however that 
$A_j \cap \ti U$ may have many components for each such 
$j$. For brevity this subsequence will still be denoted by $j$. Choose $a_j \in A_j \cap \ti U$ such that $a_j \ra a$. Let $\ti A_j$ be that component of $A_j \cap 
\ti U$ which contains $a_j$. As $A_j \subset U$ are closed it follows that $\pa \ti A_j \subset \pa \ti U$ and hence ${\rm cl}(A_j) \cap \pa \ti U \not= 
\emptyset$. Now
\[
\sup\,\{ \phi(z) : z \in {\rm cl}(A_j) \cap \pa \ti U \} < \phi(a),
\]
and hence by continuity of $\phi$ there is an open neighbourhood $V$ of ${\rm cl}(A_j) \cap \pa \ti U$ such that
\[
\sup\,\{ \phi(z) : z \in V \} = c < \phi(a).
\]
Having fixed $V$ choose an open neighbourhood $\ti V$ of $a$ such that $c < \phi(z)$ for all $z \in \ti V$. It follows that $\phi$ restricted to $\ti A_j$ attains its 
maximum in $\ti V$ which is a contradiction. Exactly the same arguments can be applied to points in ${\rm cl}(A_j) \cap M^-_s$ if any. This shows that 
${\rm cl}(A_j) \subset M \sm (M^+_s \cup M^-_s)$.

\medskip

Now pick $a \in {\rm cl}(A_j) \cap T^+_2$, if possible, and choose coordinates around $a = 0$ so that $T^+_2$ coincides with the $2$--plane spanned by $\Re z_1, \Re 
z_2$ near the origin and fix a polydisc $\ti U = \{ \vert z_1 \vert < \eta, \vert z_2 \vert < \eta \}$ around the origin with $\eta > 0$ small enough so that $\ti 
U \cap T^+_2 = \{ \Im z_1 = \Im z_2 = 0\}$. Then $\tau(z) = 2 (\Im z_1)^2 + 2 (\Im z_2)^2$ is a non-negative strongly plurisubharmonic function in $\ti U$ whose 
zero locus is exactly $\ti U \cap T^+_2$. Also note that $i \pa \ov \pa \tau = i \pa \ov \pa \vert z \vert^2$. For each $r > 0$ the domain $V_r = \{ z \in \ti U : 
\tau(z) < r \}$ is a strongly pseudoconvex tubular neighbourhood of $\ti U \cap T^+_2$. As before choose $a_j \in A_j$ converging to $a$, and by shifting, it if 
necessary, we may assume that $a_j \in A_j \sm (\ti U \cap T^+_2)$. Let $\ti A_j$ be that component of $A_j \cap \ti U$ which contains $a_j$. Fix $0 < r_0 \ll \eta$ 
whose precise value will be determined later. Then only finitely many $\ti A_j$ can be contained in $V_{r_0}$. Indeed if this does not hold pass to a subsequence, 
still retaining the same index for brevity, if necessary for which $\ti A_j \subset V_{r_0}$. Define
\[
\rho_j(z) = \tau(z) - \vert z - a_j \vert^2 / 2,
\]
and note that $\rho(a_j) = \tau(a_j) > 0$ for each $j$. Moreover,
\[
i \pa \ov \pa \rho_j = i \pa \ov \pa \tau - i \pa \ov \pa \vert z - a_j \vert^2 / 2 = i \pa \ov \pa \vert z \vert^2 - i \pa \ov \pa \vert z \vert^2 / 2 = i \pa \ov 
\pa \vert z \vert^2 > 0
\]
shows that the restriction of $\rho_j$ to $\ti A_j$ is subharmonic for all $j$. Now fix $j$ and let $w \in \pa \ti A_j \subset \pa \ti U$. Then
\[
\rho_j(w) = \tau(w) - \vert w - a_j \vert^2 / 2 \le r_0 - \vert w - a_j \vert^2 / 2 < 0
\]
the last inequality holding whenever $r_0 > 0$ is chosen to satisfy
\[
2 r_0 < \eta^2 \approx (\eta - \vert a_j \vert)^2 \le (\vert z \vert - \vert a_j \vert)^2 \le \vert z - a_j \vert^2
\]
for all $z \in \pa \ti U$. This contradicts the maximum principle and hence all but finitely many $\ti A_j$ must intersect $\ti U \cap \{ \tau(z) \ge r_0 \}$. 
Note, however, that this set contains strongly pseudoconvex points and hence it follows that ${\rm cl}(A_j) \cap M^+_s \not= \emptyset$, which is a contradiction to 
the previous step. The same reasoning can be applied to show that ${\rm cl}(A_j)$ does not lie entirely in $T^-_2 \cup T^{\pm}_1 \cup T^{\pm}_0$. What remains is 
the border $M \sm (M^+ \cup M^-)$ which as discussed earlier admits a semi-analytic stratification into real analytic submanifolds of dimension $2, 1, 0$. The top 
dimensional strata can be made maximally totally real after a possible refinement. The same arguments can be repeated for these strata to get a contradiction.
\end{proof}

\begin{prop}
Let $D, D'$ be domains in $\mbb C^2$, both possibly unbounded and $f : D \ra D'$ a non-constant holomorphic mapping. Suppose that $M \subset \pa D$ is an open, smooth real 
analytic hypersurface of finite type, and let $p \in M$. Let $U$ be a neighbourhood of $p$ in $\mbb C^2$ such that the cluster set of no point on $U \cap M$ intersects $D'$. 
Then there exists a possibly smaller neighbourhood $V$ of $p$ such that $f : V^- \ra D'$ has discrete fibres.
\end{prop}

\begin{proof}
First observe that for each $z \in D$, the analytic set $f^{-1}(f(z))$ is at most one dimensional as otherwise the uniqueness theorem will imply that $f$ is a 
constant mapping. Now if the assertion does not hold then there is a sequence $p_j \ra p$ for 
which $f^{-1}(f(p_j))$ is one dimensional at $p_j$. Let $A_j$ 
be a pure one dimensional component of $f^{-1}(f(p_j)) \cap U^-$ that contains $p_j$. Since the cluster set of no point on $U \cap M$ intersects $D'$ it follows that 
$A_j \subset U$ is a closed analytic set and hence Proposition~3.1 implies that ${\rm cl}(A_j)$ is not entirely contained in $M$. Pick $\z_0 \in {\rm cl}(A_j) \cap U^-$ and
choose $\z_j \in A_j$ converging to $\z_0$. Then $f(p_j) = f(\z_j) \ra f(\z_0) \in D'$. On the other hand, note that since $cl_f(p) \cap D' = \emptyset$, it follows that 
either $\vert f(p_j) \vert \ra + \infty$ or $f(p_j)$ clusters only at a finite boundary point of $D'$. This is a contradiction.
\end{proof}

\no Now suppose that $p \in M^+$. Then there exist (see, for example, \cite{S} and \cite{FS}) constants $\alpha, \beta > 0$ and an open neighbourhood $V$ of 
$p$ such that for every $\z \in V \cap \pa D$ there exists a plurisubharmonic function $\phi_{\z}$ on $V^-$ that is continuous on $V \cap \ov D$ satisfying
\begin{equation}
- \vert z - \z \vert^{\alpha} \lesssim \phi_{\z}(z) \lesssim - \vert z - \z \vert^{\beta}
\end{equation}
for any $z \in V \cap \ov D$. Here $(\alpha, \beta) = (1, 2)$ or $(1, 2m)$ depending on whether $p$ is strongly pseudoconvex or just weakly so, and in the latter case
$2m$ is the type of $\pa D$ at $p$. Moreover, the constants involved in these estimates are independent of $\z \in V \cap \pa D$. Thus $\phi_{\z}$ is a family of local 
plurisubharmonic barriers at $\z \in V \cap \pa D$ all of which are defined in a fixed neighbourhood of $p$. It follows from (3.1) that there are small neighbourhoods 
$V_2 \subset V_1$ of $p$ with $V_2$ compactly contained in $V_1$ and $\tau > 0$ and a smooth non-decreasing convex function $\theta$ with $\theta(t) = - \tau$ for $t \le 
- \tau$ and $\theta(t) = t$ for $t \ge -\tau / 2$ such that $\rho_p(z) = \tau^{-1} \theta(\phi_p(z)) : D \ra [-1, 0)$ is a negative continuous plurisubharmonic function on 
$D$ with $\rho_p(z) = -1$ on $D \sm V_1$ and $\rho_p(z) = \tau^{-1} \phi_p(z)$ on $V_2^-$. Since $f$ has discrete fibres near $p$ by Proposition 3.2, we may define
\[
\psi_p(z') =
\begin{cases}
  \sup\{ \rho_p(z) : z \in f^{-1}(z') \}, &\text{for $z \in f(V_1^-)$;} \\
  -1, &\text{for $z' \in D' \sm f(V_1^-)$.}
\end{cases}
\]
Arguments similar to those in \cite{Su} and \cite{ForLow} show that $\psi_p(z')$ is a negative, continuous plurisubharmonic function on $D'$. Furthermore there is an open 
neighbourhood $U'$ of $p'$ small enough so that $U' \cap M'$ is smooth real analytic such that
\begin{equation}
{\rm dist}(f(z), U' \cap M') \lesssim {\rm dist}(z, U \cap M)
\end{equation}
whenever $z \in U^-$ and $f(z) \in U'^-$. Now since $p' \in cl_f(p)$ there is a sequence $p_j \ra p$ such that $f(p_j) \ra p'$. While not much can be said at this 
stage about the continuity of $\psi_p(z')$ at $p'$, it does however follow from the definition of $\psi_p(z')$ that $\psi_p(f(p_j)) \ra 0$. This observation will be used 
in the sequel.


\section{Proof of Theorem 1.1 -- Case ($i$)}

\no In this section $p$ will be a strongly pseudoconvex point, i.e., $p \in M^+_s$, and separate cases will be 
considered depending on whether $p' \in M'^+$, $T'$, or $M'^-$ for $p' \in cl_f(p)$.

\subsection{The case when $\mathbf{ p'\in  M'^+\cup M'^-}$}
If $p' \in M'^+$, then by a theorem of Sukhov \cite{Su}, the map $f$ admits a H\"older continuous extension to a neighbourhood of $p$ on $M$. 
Further, by the result of Pinchuk-Tsyganov \cite{PT}, $f$ extends holomorphically, in fact, locally biholomorphically, across $p$. 

\medskip

Next, suppose that $p' \in M'^-_s$. Let $p_j$ be a sequence of points in $D$ such that $p_j \to p$, and $p'_j = f(p_j) \to p'$. Fix a small ball
$V' \subset \mbb C^2$ around $p'$ in which $M'$ is strictly pseudoconcave. For each $p'_j$, let $\z'_j \in M' \cap V'$ be the unique point
such that
\[
{\rm dist}(p'_j, M' \cap V') = \vert \z'_j - p'_j \vert.
\]
Let $L'_j \subset V'$ be the complex line through $p'_j$ which is obtained by translating the complex tangent space to $M'$ at $\z'_j$. Then $L'_j \ra L' \subset V'$ 
which is the complex tangent space to $M'$ at $p'$. Since $M'$ is strictly pseudoconcave at $p'$, it follows that $L' \cap \pa V' \Subset D'$ and hence that
$L'_j \cap \pa V'$ is uniformly compactly contained in $D'$ for all large $j$. Let $\phi_j : \ov \Delta \ra L_j$ be a holomorphic parametrization, which is continuous
on $\ov \Delta$ and satisfies $\phi_j(0) = p'_j$ and $\phi_j(\pa \Delta) = L'_j \cap \pa V'$. The sub-mean value property shows that
\[
\psi_p(p'_j) = \psi_p \circ \phi_j(0) \lesssim \int_{\pa \Delta} \psi_p \circ \phi_j
\]
for each $j$. Note that $\psi_p(p'_j) \ra 0$, while the right side is bounded above by a uniform negative constant since $\{ \phi_j(\pa \Delta) \}$ are uniformly
compactly contained in $D'$ for all large $j$. This is a contradiction.

\medskip

Suppose now that $p' \in T'^- = T' \cap M'^-$. Let $T'^- = T'^-_2 \cup T'^-_1 \cup T'^-_0$ be a stratification of $T'^-$ into 
totally real, real analytic manifolds of dimension 2, 1 and 0 respectively. Suppose $p'\in T'^-_2$. Let $V$ and $V'$ be small 
neighbourhoods of $p$ and $p'$. Consider the set $A=\Gamma_f \cap (V \times V')$ where $\Gamma_f$ is the graph 
of the map $f$. Then $(p,p')\in \overline A$. Since the cluster set of $M$ under $f$ does not contain points in $D'$, 
and for any point $q\in M $ near $p$, the set $cl_f(q)$ cannot contain strictly pseudoconcave points by the 
argument above, it follows that the limit points of the set $A$ in $V\times V'$ are contained in 
$(M\cap V)\times (T'^-_2\cap V')$. The latter is a real analytic CR manifold of CR dimension one. By a theorem of 
Chirka \cite{ch1} (see also \cite{ch2}), the set $(M\cap V)\times (T'^-_2\cap V')$ is a removable singularity for 
$A$, i.e., $A$ admits analytic continuation as an analytic set in $V \times V'$ after possibly shrinking these neighbourhoods 
if needed. Therefore, by \cite{DPJGA} 
the map $f$ extends holomorphically to a neighbourhood of~$p$. Arguing by induction, we may assume that 
$cl_f(p)$ does not contain points in $T'^-_2$, and repeat the argument for $T'^-_1$, and for $T'^-_0$. This 
shows that in each case $f$ admits holomorphic extension to a neighbourhood $p$. This, again, leads to a contradiction, 
because the extension will be locally biholomorphic away from a complex  analytic set of dimension one, and biholomorphic 
maps preserve the Levi form.

\subsection{The set $\mathbf{X_f}$}

The remaining possibility is 
\[
cl_f(p) \cap U' \subset T' \setminus (M'^+\cup M'^-).
\]
Note that under these conditions, there exists a sequence of points $p^j \to p$, $\{p^j\}\subset M$ such that 
$f$ extends holomorphically to a neighbourhood of each $p^j$. This follows by the previously used argument. Indeed, if 
the cluster set of a small neighbourhood of $p$ in $M$ contains strictly pseudoconvex points of $M'$ then at those 
points we have extension by \cite{Su} and \cite{PT}, which gives us points of extendability arbitrarily close to $p$.
So suppose that the limit points of the set $A=\Gamma_f \cap (V\times V')$ are contained in 
$(M\cap V)\times (T'\cap V')$ which is locally (after stratification and inductive argument on dimension) a CR 
manifold of dimension at most one. By \cite{ch2}, $A$ admits analytic continuation as an analytic set in $V \times V'$, and by 
\cite{DPJGA}, the map $f$ extends holomorphically to a neighbourhood of $p$. Then there are strongly pseudoconvex points on 
$M$ near $p$ that are mapped locally biholomorphically to strongly pseudoconcave points near $p'$ and this is a contradiction.

\medskip

Further, it is clear that for any point $p' \in cl_f(p)$, the sequence $\{p^j\}$ above can be chosen in such a way that 
$f(p^j) = p'^j \to p'$ as $j \to \infty$.
 
\begin{lem}\label{lem:X}
Let $p \in M^+_s$ and suppose that $p' \in T' \setminus (M'^+\cup M'^-)$. Then 
$X_f \subset U^+_1 \times U'^+_1$ defined by \eqref{eq:X} is a nonempty, pure two dimensional, closed analytic set. 
\end{lem}

\begin{proof}
Firstly, $X_f$ is non-empty because $f$ extends holomorphically to a neighbourhood
of $p_j$, with the extension sending $M$ to $M'$, and \eqref{eq:X} simply manifests the 
invariance property of the Segre varieties for the extension near $p_j$.

\medskip

Secondly, $X_f$ is a locally complex analytic set. Indeed, suppose that $(x,x')\in X_f$, so $f(Q_x\cap U^-)$
contains $_{{}^sx'}Q'_{x'}$.  Let $z_0\in Q_x\cap U^-$ be such that ${}^sx' = f(z_0)$. Let $V = V_1 \times V_2 \subset U^-_1$ 
be a small polydisc centred at $z_0$ such that $f(Q_x\cap V)$ is contained in the canonical component 
of $Q'_{x'}$. Such $V$ exists because $f(Q_x\cap U^-)$ contains the germ $_{{}^sx'}Q'_{x'}$ and both sets have the
same dimension. Then there exists a neighbourhood $U_x$ of $x$ such that for any
$w \in U_x$, we have $Q_w\cap V \ne \varnothing$, and moreover, we may assume that for any $z_1 \in V_1$, the point
$(z_1,h(z_1,\overline w))$ is in $Q_w \cap V$. Then the condition $f(z)\in Q'_{w'}$ for all $z\in Q_w\cap V$, 
is equivalent to $r'(f(z),\overline w)=0$, or 
\begin{equation}\label{e:r'}
r'(f(z_1, h(z_1,\overline w)),\overline w')=0,\ \ z_1\in V_1.
\end{equation}
This is an infinite system of holomorphic equations (after conjugation) that describes the property that 
$f(Q_w\cap V)\subset Q'_{w'}$ for all $w\in U_x$ and $w'\in U'^{+}_1$. By analyticity, the latter 
inclusion implies that $f(Q_w \cap U^-_1)$ contains $_{{}^sw'}Q'_{w'}$ provided that $U_x$ is sufficiently 
small. This shows that in $U_x \times U'^{+}_1$, the set $X_f$ is described by a system of holomorphic 
equations, and so $X_f$ is a locally complex analytic set.

\medskip

Thirdly, $X_f$ is closed. Indeed, let
\[
E = \big\{ z \in Q_w \cap U^-_2 : f(z) = {}^sw', f({}_zQ_w) \supset {}_{{}^sw'}Q'_{w'} \; {\rm and} \; (w, w') \in X_f  \big\}.
\]
Since $p \in M^+_s$ it follows that $Q_p \cap \ov D = \{p\}$ and hence Lemma~8.4 in \cite{DP1} shows that $E$ is relatively compactly
contained in $U_2$. Now if $(w^j, w'^j) \in X_f$ converges to $(w^0, w'^0) \in U^+_1 \times U'^{+}_1$, then we need to show that
$(w^0, w'^0) \in X_f$, i.e.,
\begin{equation}
f(Q_{w^0} \cap U^-_1) \supset  {}_{{}^sw'^0}Q'_{w'^0}.
\end{equation}
For this, choose $\z^j \in Q_{w^j} \cap U^-_2$ such that $f(\z^j) = {}^sw'^j$ and $f({}_{\z^j}Q_{w^j}) \supset {}_{{}^sw'^j}Q'_{w'^j}$. Since
$E$ is compactly contained in $U_2$, it follows that $\z^j \ra \z^0 \in U^-_2$ so that $f(\z^0) = {}^sw'^0$, and the analytic dependence of $Q_w$
on $w$ shows that
\[
f({}_{\z^0}Q_{w^0}) \supset {}_{{}^sw'^0}Q'_{w'^0},
\]
which evidently implies (4.2).
\end{proof}

Suppose that $p^{j}$ is a sequence of points on $M$ converging to
$p$ such that $f$ extends holomorphically to a neighbourhood $U_j$ of each $p_{j}$, 
and the sequence $p'^{j}=f(p^{j})$ converges to $p'\in M'$. Let $V_{j} \subset U_1$ be 
a neighbourhood of $Q_{p^{j}}$. We may choose $U_j$ and $V_j$ in such a way, that 
$U_j$ is a bidisc, and for any point $w\in V_j$, $Q_w\cap U_j$ is a non-empty connected 
graph over the $z_1$-axis. For each $j$ consider the following set:
\begin{equation}\label{eq:Xj}
X_{j}= \big\{ (w,w') \in V_j \times U'_1 : f(Q_w\cap U_j)\subset Q'_{w'} \big\},
\end{equation}
where by $f$ we mean the extension of $f$ to $U_j$. 

\begin{lem}\label{lem:Xj}
$X_j$ is a closed complex analytic subset of $V_j\times U'_1$.
\end{lem}

\begin{proof}
The proof of this lemma is similar to that of Lemma~\ref{lem:X}.
\end{proof}

Let $X_j$ be given by \eqref{eq:Xj}. We may consider only
the irreducible component of $X_j$ of dimension two that coincides near $(p^j,p'^j)$ with
the graph of $f$. For simplicity denote this component again by $X_j$. Define
\begin{equation}
C_j = X_j\cap \left( (Q_{p^j}\setminus\{p_j\})\times Q'_{p'^j}\right). 
\end{equation}
Since $M$ is strictly pseudoconvex, $\left(Q_{p^j}\setminus\{p^j\}\right)\subset U^+_1$
(provided that $U_1$ is sufficiently small), and therefore, $C_j$ is a closed complex analytic 
subset of $U^+_1\times U'_1$.


\subsection{The case when $\mathbf{ p'\in (T' \setminus (M'^+\cup M'^-))\cap T'_2}$} 
In this subsection we will concentrate on the case when $p'\in cl_f(p)$ is a point on the totally real 
2-dimensional stratum of the border separating the pseudoconvex and pseudoconcave points of $M'$. Our 
first goal is to prove the inclusion $C_j\subset X_f$. This will be done after a careful choice of the neighbourhoods in 
the target space. Choose a neighbourhood $U$ of $p$ such that all the previous conclusions are valid and let $\tilde U'\Subset U'$
be a pair of neighbourhoods of the point $p'$ such that the following hold:

\begin{enumerate}
\item [(i)] $\tilde U'$ and $U'$ are bidiscs, and for any point $w'$ in $\tilde U'^+$ the symmetric point $w'^s$ is 
contained in $U'^-$,
\item [(ii)] for any $w'\in \tilde U'$ the set $Q'_{w'}\cap U'$ is a holomorphic graph $z'_2 = h'(w',z'_1)$, and
\item [(iii)] $Q'_{w'}\cap M'$ intersects $\partial U'$ transversally for $w'\in \tilde U'$. In particular, this means that 
$Q'_{w'}\cap M'$ intersects $\partial U'$ at points near which both $Q'_{w'}\cap M'$ and $\partial U'$ are 
smooth submanifolds.
\end{enumerate}

Note that the above conditions are possible to meet because $Q'_{p'}\cap M'$ is a finite union of  isolated points and
real analytic curves with isolated singularities, and so $U'$ can be chosen to satisfy (iii) for $w'=p'$. 
Since the singularities of $Q'_{w'}\cap M'$ vary analytically with $w'$ this is also possible to achieve for all $w'$ in 
a small neighbourhood of $p'$.

\begin{lem}\label{lem:inclusion}
If $(p^j, p'^j)\in U \times \tilde U'$, then
\begin{equation}
C_j \cap (U\times \tilde U') \subset X_f \cap (U\times \tilde U').
\end{equation}
\end{lem}

\begin{proof}
Note that the inclusion holds near $(p^j,p'^j)$ because $X_f$ and $X_j$ both agree with the graph of the extension 
of $f$ to $U_j$ by the invariance property of Segre varieties. If $C_j$ is a closed analytic subset of $U^+\times U'^+$ 
(where $X_f$ is defined), then the inclusion $C_j\subset X_f$ holds by the uniqueness property for complex analytic sets. 
Thus, to prove the lemma we only need to show that $C_j$ is contained in $(U^+\times \tilde U'^+)$. Arguing by 
contradiction, suppose that $(a,a')\in C_j\cap (U^+\times M')$ for some fixed $j>0$, and that $(a,a')$ is a limit point of 
$X_f$. Let $\gamma(t)\subset Q_{p_j}$ be a smooth curve that connects $p_j=\gamma(0)$ and $a=\gamma(1)$, 
and consider a continuous family of points $a_t$ as $t$ varies from 0 to 1. Suppose that for all points on this curve except 
the terminal point $a$, we have $(a_t, a'_t) \in X_f$ for some point $a'_t \in U'$, i.e., the curve $\gamma$ is 
contained in the projection of $X_f$ to the first component. We will need the following lemma.

\begin{lem}\label{lem:1.1}
For any $t$, $0 \le t < 1$ the set $f(Q_{a_t}\cap U^-)\cap U'^-$ contains a connected component,
which we denote by $Z_t$, such that $p'_j\in \overline Z_t$, and ${}^sa'_t\in Z_t$. In particular, there is a path 
$\tau_t \subset Z_t\cup \{p'_j\}$ starting at $p'_j$ and terminating at ${}^sa'_t$.
\end{lem}

\begin{proof}[Proof of Lemma~\ref{lem:1.1}] Note here that the point $p'_j$ and the path $\gamma'$ 
(consisting of points $a'_t$) is contained in $\tilde U'$. Let $R$ be the subset of $[0,1)$ for which the lemma holds. 
Then $R\ne\varnothing$, because it holds for sufficiently small $t$.

\medskip

\no {\it Claim 1:} $R$ is an open set. Indeed, if $t_0\in R$, then since all the data is analytic, and the path $\tau_{t_0}$
is compactly contained in $U'$, a small perturbation of $t$ near $t_0$ will preserve the property described by the lemma.

\medskip

\no {\it Claim 2:} If $[0,t_0)\subset R$, then $t_0\in R$. Indeed, let $U'=U'_1\times U'_2\subset \mathbb C_{z'_1}\times \mathbb C_{z'_2}$, 
and let $P: U' \to U'_1$ be the coordinate projection. Consider the sets $P(Z_t)$ for $t<t_0$. Then $P(Z_t)$ is a connected 
open set. Its boundary consists of points from the boundary of $U'_1$ and the interior points of $U'_1$. The latter are
the projections of the points in the closure of $f(Q_{a_t}\cap U^-)\cap U'^-$ that are contained in $M'$. To see this 
observe that $f(Q_{a_t}\cap U^-)\subset Q'_{a'_t}$, and for any $w'$, the restriction $P|_{Q'_{w'}}$ is a biholomorphic 
map because $Q'_{w'}$ is a graph over $U'_1$. By assumption, the point $P(p'_j)$ is on the boundary of $P(Z_t)$, and 
$P({}^sa'_t)$ is an interior point of $P(Z_t)$ for $t<t_0$.  We denote by $P(Z_{t_0})$ the limit set of the sequence 
$P(Z_t)$ as $t\to t_0$, which by construction is a subset of $P(f(Q_{a_{t_0}}\cap U^-)\cap U'^-)$. We will show 
that $P(Z_{t_0})$ has the same property, which will complete the proof of Claim 2.  

\medskip

The claim may fail only if the set $P(Z_{t_0})$ is disconnected, and the point $P(p'_j)$ will 
not be on the boundary of the connected components of $P(Z_{t_0})$ that contains $P({}^sa'_{t_0})$. Suppose 
this is the case. Then since at any time $t<t_0$, the set $P(Z_t)$ is connected, 
it follows that the connected components of $P(Z_{t_0})$ containing $P(p'_j)$ (on the boundary) and 
$P({}^sa'_{t_0})$ must have at least one common boundary point. Denote by $S_1$ the component of $P(Z_{t_0})$ 
that contains $P(p'_j)$ on the boundary, and by $S_2$ the component that contains $P({}^sa'_{t_0})$. Two cases are 
possible.

\medskip

(i) $S_1$ and $S_2$ have common boundary points only in $\partial U'_1$.

(ii) $S_1$ and $S_2$ have at least one common boundary point in $U'_1$.

\medskip

Suppose (i) holds. Then $Q'_{a'_{t_0}}\cap M'$ intersects $\partial U'$ either tangentially, or at a singular point, 
which is not allowed by the choice of $\tilde U'$.

Suppose now (ii) holds, and $\zeta'_1\in \overline S_1 \cap \overline S_2 \cap U'_1$. Then as discussed above, 
$\zeta'_1$ is the projection of a point  $\zeta'=(\zeta'_1,\zeta'_2)\in M'$ which is in  the closure of 
$f(Q_{a_{t_0}}\cap U^-)\cap U'^-$. Because only points 
from the boundary of $D$ can be mapped by $f$ into the boundary of $D'$, we conclude that for any sequence of 
points in $P^{-1}(S_1)\cap Z_{t_0}$ converging to $\zeta'$, the sequence of preimages under $f$ converges to 
$Q_{a_{t_0}}\cap M$. By passing to a subsequence we may assume that the latter sequence converges to a 
point $\zeta \in Q_{a_{t_0}}\cap M$. We show that $f$ extends holomorphically to a neighbourhood  of $\zeta$. 
This can be proved as follows.

\medskip

There is a path $$\sigma' \subset f(Q_{a_{t_0}}\cap U^-)\cap U'^- $$ that connects $p'_j$ and $\zeta'$. It 
can be obtained, for example, by taking a path in the closure of  $S_1$ connecting $P(p'_j)$ and $\zeta'_1$, and 
lifting it to the closure of $f(Q_{a_{t_0}}\cap U^-)\cap U'^-$.  Take  $\sigma=f^{-1}(\sigma')$, more precisely, 
let $\sigma$ be the component of $f^{-1}(\sigma')$ that contains $p_j$. We claim that the germ of the map $f$ at $p_j$ 
extends analytically to a neighbourhood of the closure of~$\sigma$. 
For the proof, we choose a small neighbourhood $U_{a_{t_0}}$ of the point $a_{t_0}$, and a thin neighbourhood 
$V_{t_0}$ of $Q_{a_{t_0}}\cap U$. If $U_{a_{t_0}}$ is small enough, then over $U_{a_{t_0}}$ the set $X_j$ is 
a ramified covering, in particular, it defines a holomorphic correspondence $F_{t_0}: U_{a_{t_0}}\to U'$. 
Now choose $V_{t_0}$ such that for any $w\in V_{t_0}$, the set $Q_w\cap U_{a_{t_0}}$ is nonempty and 
connected, and consider the set
$$Y_{t_0} = \big\{ (w,w')\in V_{t_0} \times U' : F_{t_0}(Q_w\cap U_{a_{t_0}})\subset Q'_{w'} \big\}.$$
An argument similar to that of Lemma~\ref{lem:X}  (cf. \cite{Sh}) shows that $Y_{t_0}$ is a nonempty complex 
analytic set. Further, there exists an irreducible component of $Y_{t_0}$ which agrees with the graph of $f$ near 
$(p_j, p'_j)$. Indeed, let $z$ be any point in $U_j\cap V_j\cap V_{t_0}$, where $U_j$ and $V_j$ are the 
neighbourhoods from \eqref{eq:Xj}. Let $w\in Q_z\cap U_{a_{t_0}}$ be an arbitrary point. It follows that 
$z\in Q_w$. Let $w'\in F_{t_0}(w)$. This means by definition of $F_{t_0}$ that $f(Q_w\cap U_j)\subset Q'_{w'}$, 
in particular, $f(z)\in Q'_{w'}$, or $w'\in Q'_{f(z)}$. From this we conclude that $F_{t_0}(w)\subset Q'_{f(z)}$. 
Since $w$ was an arbitrary point of $Q_z\cap U_{t_0}$, it follows that 
$F_{t_0}(Q_z\cap U_{a_{t_0}})\subset Q'_{f(z)}$, which means that $(z,f(z))\in Y$. We will consider only this
irreducible component of $Y_{t_0}$, which for simplicity we denote again by $Y_{t_0}$.

\medskip

It follows that the same inclusion holds in a neighbourhood of the curve $\sigma$. By continuity and using the fact 
that $f(\sigma)=\sigma'$, we have that $Q_{\zeta}\cap U_{a_{t_0}}$ is mapped by $F_{t_0}$ into 
$Q'_{\zeta'}$, and therefore, we actually have the set $Y_{t_0}$ defined in a neighbourhood of $\zeta$, which 
gives the extension of $f$ to a neighbourhood of the point~$\zeta$. Let $\tilde f$ be the extension of $f$ near $\zeta$. 
Then $\tilde f(\zeta)=\zeta'$. Since $\zeta \in M^+_s$, 
by the invariance of the Levi form, $\zeta'$ can only be a strictly pseudoconvex point of $M'$, and therefore, 
$\tilde f$ is locally biholomorphic in some neighbourhood $V$ of $\zeta$. By the invariance of the Segre varieties, we 
have 
$$\tilde f (V\cap Q_{a_{t_0}}) = \tilde f(V) \cap Q'_{a'_{t_0}}.$$ 
But now we reach a contradiction: near $\zeta$ the set $Q_{a_{t_0}}\cap U^-$ is connected, while near 
$\zeta'$  the set $Q'_{a'_{t_0}}\cap U'^-$ has at least two components. Thus, case (ii) is also not possible, and this 
completes the proof of Lemma~\ref{lem:1.1}.
\end{proof}

We continue with the proof of Lemma~\ref{lem:inclusion}. For a fixed $j$ we are interested in understanding the 
points on $C_j$ on $(U^+ \times M')$ that are limit points for $X_f$. Let $(a,a')$ and $\gamma(t)$ be as at the 
beginning of the proof of the lemma. Observe that in this case $a'\in T'_2$. Indeed, if $a'\in M'^+_s$, then for  $t$
close to 1, the set $Q'_{a'_t}\cap U'^-$ has a small connected component near $a'$, which contains ${}^sa'_t$. These
components shrink to the point $a'$ as $t$ approaches 1. But this contradicts Lemma~\ref{lem:1.1} which states 
that there exists a component $Z_t \subset f(Q_{a_t} \cap U^-) \cap U'^-$ containing $p'_j$ in its closure and 
${}^sa'_t$ for all $0 \le t < 1$. Suppose now that $a'\in M'^-_s$. Then for all $t<1$, the set $Q_{a_t}\cap U^-$ 
contains a point which is mapped by $f$ to ${}^sa'_t$. Since ${}^sa'_t$ approach $a'$ at $t\to 1$, we conclude that 
there exists a point on $Q_{a}\cap M$ whose cluster set under the map $f$ contains $a'$ -- a strictly pseudoconcave 
point. However, by the previous considerations we know that this is not possible.

\medskip

Let $(a, a') \in C_j \cap (U^+ \times T'_2)$ be a limit point for $X_f$. By Lemma~\ref{lem:1.1} there exists a curve
$t \mapsto (a_t, a'_t) \in (Q_{p_j} \times Q'_{p'_j}) \cap (U^+ \times U'^+)$ parametrized by $[0, 1]$ that 
converges to $(a, a')$, and  a component $Z_t \subset f(Q_{a_t} \cap U^-) \cap U'^-$ containing $p'_j$ 
in its closure and ${}^sa'_t$ for all $0 \le t < 1$. For each $0 \le t < 1$ let $\sigma'_t$ be a path in 
$Z_t \cup \{p'_j\}$ that joins $p'_j$ and ${}^sa'_t$. Let $\Omega'$ be a tubular neighbourhood of $T'_2$ in $\mathbb C^2$ 
chosen so small that it does not contain $p'_j$. As $t \rightarrow 1$, $a'_t \rightarrow a' \in T'_2$, and hence the symmetric points 
${}^sa'_t \rightarrow a'$ which implies that ${}^sa'_t \in \Omega'$ for $t$ close to $1$. Since $\sigma'_t$ join 
$p'_j$ and ${}^sa'_t$, it follows that these paths must leave $\Omega'$ for each $0 \le t < 1$. Let $\gamma'_t$ 
be the component of $\sigma'_t \cap \Omega'$ that contains ${}^sa'_t$. Then the other end point of 
$\gamma'_t$ lies on $\partial \Omega'$. It is possible to choose a subsequence of $\gamma'_t$ that converges 
in the Hausdorff metric to a continuum, say $\gamma'_1$, which is contained in $\overline \Omega'$ and which 
contains $a'$ and points on $\partial \Omega'$. Furthermore each $\gamma'_t$ is contained in 
$f(Q_{a_t} \cap U^-) \subset Q'_{a'_t}$ by the invariance property and hence the limiting continuum 
$\gamma'_1$ must lie in $Q'_{a'}$. Recall that $Q'_{a'} \cap M'$ is a finite union of smooth real analytic arcs and 
singular points by the choice of $\tilde U'$ and $U'$. Moreover the various components of $Q'_{a'} \cap U'^-$ have their boundaries 
contained in the union of these real analytic arcs and points and each component contains points on $T'_2$ in its 
closure. Two cases arise:

\medskip

\noindent (i) Suppose that $\gamma'_1$ contains a point $z'_0 \in \partial \Omega' \cap D'$. Let $C'$ be the 
component of $Q'_{a'} \cap D'$ which contains $z'_0$. Since $f : Q_a \cap D \rightarrow D'$ is proper it follows 
that the image $f(Q_a \cap D)$ is a closed subvariety in $D'$. Now $f(Q_a \cap D)$ contains $z'_0$ (in fact points 
on $\gamma'_1$ near $z'_0$ are also contained in $f(Q_a \cap D)$ and this implies that $C' \subset f(Q_a \cap D)$. 
Let $\Gamma'$ be a boundary component of $C'$; evidently $\Gamma' \subset Q'_{a'} \cap M'$. Pick a point 
$z' \in \Gamma' \cap T'_2$ and join it to $z'_0$ by a path $r(t)$ that lies in $C' \cup \{z'\}$. Then extend 
$f^{-1}$ analytically along $r(t)$. This process gives rise to an extension of $f^{-1}$ at $z'$ and it means that points 
on $T'_2$ are mapped to strongly pseudoconvex points. Contradiction. 

\medskip

\noindent (ii) Suppose that $\gamma'_1$ is contained entirely in $Q'_{a'} \cap M'$. It follows
that $\gamma'_1$ must be contained in the connected union of one or more of the arcs that make up 
$Q'_{a'} \cap M'$. Let $\Gamma'$ be one such smooth real analytic arc that is contained in $\gamma'_1$ and 
which contains points on $\partial \Omega' \cap M'$ as well as those on $T'_2$. Repeating the same argument as 
in case (i) it is possible to get a contradiction for exactly the same reason.

\medskip

This completes the proof of Lemma~\ref{lem:inclusion}.
\end{proof}

\no Since $X_f \subset U^+\times U'^+$, Lemma~\ref{lem:inclusion} immediately implies the following:

\begin{cor}\label{cor:C}
If $(p^j, p'^j)\in U \times \tilde U'$, then $C_j\cap \left(U^+\times U'^-\right) = \varnothing$.
\end{cor}

\begin{lem}\label{lem:2}
Let $X_f$ and $C_j$ be defined as above. Then at least one of the following statements holds:
\begin{enumerate}
\item[(i)] $\overline X_f \cap (U^+ \times (U' \cap M'))=\varnothing$.
\item[(ii)] The cluster set of $\{C_j\}$ cannot be $\{p\}\times Q'_{p'}$.
\end{enumerate} 
\end{lem}

\begin{proof}
Suppose (i) does not hold, $(\zeta,\zeta')\in \overline X_f \cap (U\times (U' \cap M'))$, and 
$(\zeta^j,\zeta'^j)\to (\zeta,\zeta')$ as $j\to\infty$, $(\zeta^j,\zeta'^j)\in X_f$. Then $\zeta'$ 
cannot be a strictly pseudoconvex point, as otherwise, $f:D\to D'$ would have values outside
$D'$. Also $\zeta'$ cannot be a strictly pseudoconcave point by previous considerations. So 
we conclude that $\zeta'\in T'_2$. We may as well assume that $\zeta'=p'$. By analyticity,
the set $f(Q_\zeta\cap U^-)$ is the limit of $f(Q_{\zeta^j}\cap U^-)$. Therefore, we conclude 
that $f(Q_\zeta\cap U^-)\cap U'^-$ must be contained in $Q'_{p'}\cap U'^-$.  In particular, this means 
that $Q'_{p'}\cap U'^-$ is not empty. 

\medskip

Suppose that $\{p\}\times Q'_{p'}$ is the cluster set of $C_j$ as $j\to\infty$. Since 
$Q'_{p'}\cap U'^-\ne \varnothing$, we 
conclude that $C_j\cap \left(U^+\times U'^-\right) \ne\varnothing$ for sufficiently large $j$. 
But this contradicts Corollary~\ref{cor:C}.
\end{proof}

\noindent We now show that $f$ extends holomorphically across $p$ in either of the possibilities listed in Lemma 4.6.
Indeed, suppose that Lemma~\ref{lem:2} (ii) holds. It follows that there exists a point $(a,a')\in X_f$ such that 
$a\in Q_p\setminus{p}$, and $\pi(X_f)$ contains some open neighbourhood $U_a$ of $a$ in $U$. We show that 
$f$ extends holomorphically to a neighbourhood of $p\in M$. Our argument is similar to the argument used in 
Lemma~\ref{lem:inclusion}. We choose neighbourhoods $U_a$ and $U'_{a'}$ of $a$ and $a'$ respectively 
in such a way that the projection $\pi : X_f\cap (U_a\times U'_{a'})\to U_a$ is a ramified covering, 
which gives rise to a holomorphic correspondence $F_a: U_a\to U'_{a'}$. Let $V_a$ be a neighbourhood of 
$Q_a$ such that for any $w\in V_a$, $Q_w\cap U_a$ is a non-empty connected set. Note that $p\in V_a$ 
because $a\in Q_p$. Define the set 
\begin{equation}
Y=\big\{(w,w')\in V_a\times U': F_a(Q_w\cap U_a)\subset Q'_{w'} \big\}.
\end{equation}
As before, $Y$ is a closed complex analytic subset of $V_a\times U'$. We now show that 
$Y\ne\varnothing$ by showing that $Y$ contains a piece of the graph of $f$. Let $j$ be sufficiently 
large such that $U_j\cap V_a\ne\varnothing$. We may fix such $j$ and assume further that $U_j\subset V_a$, 
and $V_j\cap U_a\ne\varnothing$ (this is possible because $Q_{p^j}\to Q_p$). Let $z\in U_j$, 
and let $w\in Q_z\cap U_a\cap V_j$ be arbitrary.
Let $w'\in F_a(w)$. Then by definition of $X_j$, we have $f(Q_w\cap U_j)\subset Q'_{w'}$, in
particular, $f(z)\in Q'_{w'}$. This implies that $w'\in Q'_{f(z)}$. Since $w$ was an arbitrary point
in $Q_z\cap U_a\cap V_j$, and $w'$ was any point in $F_a(w)$, it follows that 
$$
F_a(Q_z\cap U_a\cap V_j)\subset Q'_{f(z)}.
$$
Since $Q_z\cap U_a$ is a connected set, of which $Q_z\cap U_a\cap V_j$ is a non-empty open subset, we 
conclude that $F_a(Q_z\cap U_a)\subset Q'_{f(z)}$, which means that $(z,f(z))\in Y$. This proves the 
claim. Using standard arguments, one can conclude that the set $Y$ gives a holomorphic extension of 
$f$ to a neighbourhood of any point in $Q_a\cap M$, in particular, to a neighbourhood of $p$. 

\medskip

Thus, the remaining case to consider is (i) in Lemma~\ref{lem:2}, i.e., when 
$\overline X_f \cap (U\times M')=\varnothing$. For this, first observe that:

\begin{lem}
$X_f$ has no limit points on $(U \cap M) \times \ti U'^+$. 
\end{lem}

\begin{proof}
Suppose that $(w^0, w'^0) \in \ov X_f \cap \big((U \cap M) \times \ti U'^+ \big)$ and let  $(w^j, w'^j) \in X_f \subset U^+ \times \ti U'^+$ 
be such that $(w^j, w'^j) \ra (w^0, w'^0)$. Then
\[
f(Q_{w^j} \cap D) \supset {}_{{}^sw'^j}Q'_{w'^j}
\]
holds for all $j$. Choose $\z^j \in Q_{w^j} \cap D$ such that $f(\z^j) = {}^sw'^j$. By continuity ${}^sw'^j \ra {}^sw'^0 \in D'$. The
strict pseudoconvexity of $U \cap M$ implies that $Q_{w^j} \cap D$ shrinks to $w^0$ as $j \ra \infty$ and therefore $\z^j \ra w^0$. This 
contradicts (3.2).
\end{proof}
 
\begin{lem}
$\ov C_j \subset U \times \ti U'$ is a closed analytic set of pure dimension one for each $j$.
\end{lem}

\begin{proof}
By Lemma 4.3, $C_j \subset X_f$ and hence $C_j$ has no limit points on $(U \cap M) \times \ti U'^+$. Moreover there are no limit points for $C_j$ 
on $U^+ \times (U' \cap M')$ either by the assumption that $X_f$ has no limit points there. Therefore, 
$\ov C_j \sm C_j \subset \{p^j\} \times (Q'_{f(p^j)} \cap M')$. Suppose that $(p^j, q) \in \ov C_j \sm C_j$ and choose $(w^k, w'^k) \in C_j$ converging to
$(p^j, q)$ as $k \ra \infty$. Then
\[
f(Q_{w^k} \cap D) \supset {}_{{}^sw'^k}Q'_{w'^k}
\]
for all $k$, and let $\z^k \in Q_{w^k} \cap D$ be such that $f(\z^k) = {}^sw'^k$. By continuity, 
${}^sw'^k \ra {}^sq = q$, and since $U \cap M$ is 
strictly pseudoconvex, it follows that $\z^k \ra p^j$ as $k \ra \infty$. However, $f$ extends across $p^j$ and hence $f(\z^k) \ra q$. It follows that
$q = f(p^j) = p'^j$. This shows that $(p^j, p'^j)$ is the only limit point for $C_j$ and by the Remmert-Stein theorem, 
$\ov C_j \subset U \times \ti U'$ is a closed complex analytic set of pure dimension one for each $j$.
\end{proof}

\begin{lem}
The cluster set of $\{C_j\}$ is non-empty in $U^+ \times \ti U'^+$. 
\end{lem}

\begin{proof}
For $\ep > 0$ small, let $U_{\ep} \Subset U$ and $U'_{\ep} \Subset \ti U'$ be bidiscs of size $\ep$ around $p, p'$ whose sides are parallel to those of $U$ and $\ti 
U'$ respectively. Consider the non-empty analytic sets $\ov C_j \cap (U_{\ep} \times U'_{\ep})$ and examine the coordinate projection
\[
\pi' : \ov C_j \cap (U_{\ep} \times U'_{\ep}) \ra U'_{\ep}.
\]
There are two cases to be considered.

\medskip

\no {\it Case (i):} If $\pi'$ is proper for all large $j$, then $\pi'\big(\ov C_j \cap (U_{\ep} \times U'_{\ep}) \big) = Q'_{p'_j} \cap U'_{\ep}$. In this case it 
follows that $Q'_{p'} \cap U'_{\ep}$ cannot intersect $D'$. Indeed, if possible, choose $\tau'^0 \in Q'_{p'} \cap U'^-_{\ep}$. Choose $\tau'^j \ra Q'_{p'^j} \cap 
U'_{\ep}$ such that $\tau'^j \ra \tau'^0$. Then $\tau'^j \in U'^-_{\ep}$ for $j$ large which contradicts Lemma 4.3, according to which $\ov C_j$ does not contain 
points over $\ti U'$.  For $\eta < \ep$, let $U_{\eta}, U'_{\eta}$ be bidiscs around $p, p'$ respectively, of size $\eta$. Pick 
$w'^0 \in Q'_{p'} \cap \pa U'^+_{\ep/2}$ and let $w'^j \in Q'_{p'_j} \cap \pa U'^+_{\ep/2}$ be such that $w'^j \ra w'^0$. Since $\pi'$ 
is proper, choose $w^j \in Q_{p^j}$ such that $(w^j, w'^j) \in \ov C_j \cap (U_{\ep} \times U'_{\ep})$. After passing to a subsequence we may assume that $(w^j, w'^j) 
\ra (w^0, w'^0)$ where $w^0 \in Q_p$. Now Lemma~4.7 shows that $w^0 \notin U \cap M$ since $w'^0 \in U'^+_{\ep/2}$. Hence $w^0 \in U^+_{\ep/2}$, and therefore $(w^0, 
w'^0) \in {\rm cl}(C_j) \cap (U^+ \times \ti U'^+)$.

\medskip

\no {\it Case (ii):} If for some subsequence still indexed by $j$, the projection $\pi'$ is not proper, then it is possible to choose $(w^j, w'^j) \in \ov C_j \cap 
(U_{\ep} \times U'_{\ep})$ with $w^j \in Q_{p^j} \cap \pa U_{\ep}$. If for some fixed $\eta > 0$
\[
Q'_{p'^j} \cap U'_{\eta} \subset \pi' \big( \ov C_j \cap (U_{\ep} \times U'_{\ep}) \big)
\]
for all large $j$, then this is exactly the situation addressed in Case (i) above and hence we may assume without loss of generality that $w'^j \ra p'$. The strict 
pseudoconvexity of $U \cap M$ implies that $Q_{p^j} \cap \pa U_{\ep} \Subset U^+$ uniformly and hence $w^j \ra w^0 \in U^+$ after passing to a subsequence. Thus 
$(w^0, p') \in U^+ \times (U' \cap M')$ is a limit point for $X_f$ which is a contradiction.
\end{proof}

\no By Theorem~7.4 of \cite{DP2} it is known that the volumes of $\{C_j\} \subset U^+ \times \ti U'^+$ are uniformly bounded on each compact subset of $U^+ \times \ti 
U'^+$ after perhaps passing to a subsequence. By Bishop's theorem, $C_j$ converges to a pure one dimensional analytic set, say $C_p \subset U^+ \times \ti U'^+$. 
Since $C_j \subset Q_{p^j} \times Q'_{p'^j}$, it follows that $C_p \subset Q_p \times Q'_{p'}$. In particular there are points on $Q_p \sm \{p\}$ over which $X_f$ is 
defined. This is exactly the situation considered in Case (ii) of Lemma 4.6 and the arguments presented there show that $f$ extends holomorphically across $p$. 
Consequently when $p \in M^+_s$ and $p' \in (M' \sm (M'^+ \cup M'^-)) \cap T'_2$, $f$ extends holomorphically across $p$ and $f(p) = p'$. The invariance property of 
Segre varieties shows that (cf. \cite{DP1}) $f^{-1}$ extends across $p'$ as a holomorphic correspondence and thus there are strictly pseudoconcave 
points near $p'$ that are mapped locally biholomorphically by some branch of $f^{-1}$ to strictly pseudoconvex points near $p$ and this is a contradiction. This completes the 
discussion in case $p'$ is on a two dimensional totally real stratum of the border between the pseudoconvex and pseudoconcave points.

\subsection{The case when $\mathbf{ p'\in (T' \setminus (M'^+\cup M'^-))\cap (T'_1\cup T'_0)}$}

Exactly the same arguments can be applied when $p'$ is on a one dimensional stratum of the border -- the proof uses the additional fact that we know from the above 
reasoning, i.e., the cluster set of a strongly pseudoconvex point cannot intersect a totally real two dimensional stratum of the border. The case when $p'$ is a point on the 
zero dimensional stratum of the border goes as follows -- observe that the cluster set of a strongly pseudoconvex point cannot intersect either a two or one dimensional 
stratum of the border. So if $p' \in cl_f(p)$ is on the zero dimensional stratum, it must be isolated in the cluster set of $p$ and hence $f$ is continuous up to $M$ near $p$. Therefore, by \cite{DP2}, $f$ admits a holomorphic extension across $p$ with $f(p) = p'$. This is a contradiction as explained before.


\section{Proof of Theorem 1.1 -- Case ($ii$)}

\no The cases to be considered are $p \in T^+_2 \cup T^+_1 \cup T^+_0$ for the other possibility that $p \in M^+_s$ is covered by the previous section. 

\subsection {The case when $\mathbf{p \in T^+_2}$} 
Recall that $cl_f(p) \subset M'$ and hence a point $p'\in cl_f(p)$ could belong to either $M'^+, T'$ or $M'^-$. 
The arguments when $p' \in M'^{\pm}$ are similar to those used in Section~4, and therefore we shall be brief in these cases. Indeed, when $p' \in M'^+$ then by \cite{Su} 
the map $f$ admits a H\"{o}lder continuous extension to a neighbourhood of $p$ on $M$ and hence by \cite{DP2} it follows that $f$ extends holomorphically across 
$p$. In case $p' \in M'^-_s$, the same argument from Section 4 applies without any changes -- indeed, the main ingredient there is the negative, continuous 
plurisubharmonic function $\psi_p(z')$ on $D'$ which can be constructed in this case as well because $p \in T^+_2$.

\medskip

Suppose now that $p'\in T'^- = T'\cap M'^-$. Let $T'^- = T'^-_2 \cup T'^-_1 \cup T'^-_0$ be a stratification of $T'^-$ into totally real, real analytic manifolds of 
dimensions $2, 1$ and $0$ respectively. Suppose that $p'\in T'^-_2$ and let $V, V'$ be small neighbourhoods of $p, p'$ respectively. Evidently $A = \Gamma_f \cap (V \times 
V')$, where $\Gamma_f$ is the graph of $f$, contains $(p, p')$ in its closure. Then $(\ov A \sm A) \cap (V \times V')$ cannot be contained in $(T'^+_2 \times T'^-_2) \cap (V 
\times V')$, for if not, $A$ will admit analytic continuation as an analytic set across the totally real manifold $T^+_2 \times T'^-_2$. Hence by \cite{DPJGA}, $f$ will extend 
holomorphically to a neighbourhood of $p$. Proceeding by induction, we may assume that $cl_f(p)$ does not contain points in $T'^-_2$ and repeat the argument for the lower 
dimensional strata. Thus in each case $f$ admits holomorphic extension to a neighbourhood of $p$. This is a contradiction because the extension will be locally biholomorphic 
away from a codimension one analytic set, and biholomorphic maps preserve the Levi form.

\medskip

The remaining possibility is that
\[
cl_f(p) \subset T'\sm (M'^+ \cup M'^-),
\]
in which case the arguments used above show that for every $p'\in cl_f(p)$, there is a sequence $p^j \ra p$, $\{p^j\} \subset M$ such that $f$ extends holomorphically 
to a neighbourhood of each $p^j$ and $f(p^j) \ra p' \in T'\sm (M'^+ \cup M'^-)$. To deal with this case, let $U'\subset \mbb C^2$ be an open neighbourhood that compactly 
contains $cl_f(p) \subset M'$ and such that $U'\cap M'$ is a closed, smooth real analytic hypersurface of finite type. We may also assume that $U'$ is small enough to 
guarantee the existence of $Q'_{w'}$ as a local complex manifold for $w'\in U'^+$. Having chosen such a $U'$, fix a standard pair of neighbourhoods $U_1 \subset U_2$ around 
$p$ so that $(Q_p \sm \{p\}) \cap M \cap U_2 \subset M^+_s$ -- this is possible by Lemma 12.1 in \cite{DP1}, and such that $f(U^-_2)$ is compactly contained in $U'$. 
This latter condition can be fulfilled since $cl_f(p) \subset M'$. By shrinking $U_2$ further if needed, we may additionally assume that for $w'\in \pa U'\sm D'$, the 
symmetric point ${}^sw' \notin f(U^-_2)$. This is possible since 
$$
{\rm dist}(w', M') \backsimeq {\rm dist}({}^sw', M')
$$
for all $w'$ in a given compact set in $\mbb C^2$ that intersects $M'$. Then
$$
X_f = \big\{ (w, w') \in U^+_1 \times U'^+ : f(Q_w \cap D) \supset {}_{{}^sw'}Q'_{w'}  \big\}
$$
is closed by Lemma~12.2 in \cite{DP1} and also complex analytic by the arguments used before in Lemma~4.1. Furthermore, $X_f$ is non-empty because of the existence of 
the sequence $p^j \ra p$ such that $f$ extends across $p^j$ mentioned above.

\begin{lem}
$X_f$ does not have limit points on $U^+_1 \times (\pa U' \sm D')$.
\end{lem}

\begin{proof}
Suppose that $(w^0, w'^0)$ is a limit point for $X_f$ on $U^+_1 \times (\pa U'\sm D')$ and let $(w^j, w'^j) \in X_f$ converges to $(w^0, w'^0)$. Then
\[
f(Q_{w^j} \cap D) \supset {}_{{}^sw'^j}Q'_{w'^j}
\]
holds for all $j$, and choose $\z^j \in Q_{w^j} \cap D$ such that $f(\z^j) = {}^sw'^j$. After passing to a subsequence $\z^j \ra \z^0$ for some $\z^0$ in the closure of $U_2 \cap D$. This is because 
\[
E = \big\{ z \in Q_w \cap U^-_2 : f(z) = {}^sw', f({}_zQ_w) \supset {}_{{}^sw'}Q'_{w'} \; {\rm and} \; (w, w') \in X_f \big\}
\]
is closed. Evidently, ${}^sw'^j \ra {}^sw'^0$. Note that ${}^sw'^0$ is contained in the cluster set of $\z^0$, which is a contradiction since 
${}^sw'^0 \notin f(U^-_2)$ by construction.
\end{proof}

\no Define
\[
\cal L = \bigcup_{w \in T^+_2} Q_w,
\]
where the neighbourhoods $U_1, U_2$ are small enough so that $(Q_p \sm \{p\}) \cap M \cap U_2 \subset M^+_s$.

\begin{lem}
All limit points of $X_f$ on $U^+_1 \times (U' \cap M')$ are contained in $\cal L \times (T' \sm (M'^+ \cup M'^-))$. 
\end{lem}

\begin{proof}
Let $(w^0, w'^0)$ be a limit point for $X_f$ in $U^+_1 \times (U' \cap M')$, and suppose that $(w^j, w'^j) \in X_f$ converges to $(w^0, w'^0)$. Then
\begin{equation}
f(Q_{w^j} \cap D) \supset {}_{{}^sw'^j}Q'_{w'^j}
\end{equation}
holds for all $j$, and let $z^j \in Q_{w^j} \cap D$ be such that $f(z^j) = {}^sw'^j$. Note that ${}^sw'^j \ra {}^sw'^0 = w'^0$ and since $E$ (as defined in the previous lemma) 
is compactly contained in $U_2$, it follows that $z^j \ra z^0 \in Q_{w^0} \cap U_2 \cap M$. In particular, $z^0 \in M^+_s$ or $T^+_2$, and $w'^0 \in cl_f(z^0)$. If $w'^0 \in 
M'^+$, then $f$ holomorphically extends to a neighbourhood, say, $\Om$ of $z^0$ and $f(z^0) = w'^0$. It is therefore possible to choose $\z^j$ close to $z^0$ such that 
$f(\z^j) = w'^j$. The invariance property of Segre varieties shows that
\[
f(Q_{\z^j} \cap \Om) \supset {}_{{}^sw'^j}Q'_{w'^j},
\]
which when combined with (5.1) shows that $Q_{\z^j} = Q_{w^j}$. By passing to the limit, we get $Q_{z^0} = Q_{w^0}$. This is evidently a contradiction since $I_{w^0} = 
\la^{-1}(\la(w^0)) \subset M$.

\medskip

The case $w'^0 \in M'^-$ does not arise as seen before. Hence the only possibility is that $w'^0 \in T' \sm (M'^+ \cup M'^-)$. In this case, $z^0 \notin M^+_s$ again by the 
results of Section 4 and hence $z^0 \in T^+_2$. Consequently, $w^0 \in Q_{z^0} \subset \cal L$.
\end{proof}

\begin{lem}
$\cal L$ is everywhere a finite union of smooth real analytic three dimensional submanifolds of $U_2$. In particular, at all of its smooth points, the CR dimension of $\cal L$ 
is one.
\end{lem}

\begin{proof}
It is possible to choose coordinates around $p = 0$ so that $T^+_2$ becomes the totally real plane $i \mbb R^2 \subset \mbb C^2$. The defining function for $M$ near $p = 0$ can 
then be written as
\[
r(z, \ov w) = 2 x_2 + (2 x_1)^{2m} a(z_1, y_2)
\]
where $z_1 = x_1 + i y_1, z_2 = x_2 + i y_2$ with $m > 1$, and $a(z_1, y_2)$ a real analytic function which is positive near the origin. Recall that the complexification of $M$ 
is given by $r(z, \ov w)$ where $(z, w) \in U_2 \times U_2$. Let $U^{\ast}_2$ denote the open set $U_2$ equipped with the conjugate holomorphic structure. Then
\[
M^{\mbb C} = \big\{ (z, w) \in U_2 \times U^{\ast}_2 : r(z, \ov w) = 0 \big\}
\]
is a smooth, closed complex manifold in $U_2 \times U^{\ast}_2$ of dimension $3$. Note that
\[
r(z, \ov w) = z_2 + \ov w_2 + (z_1 \ov w_1)^{2m} \tilde a(z_1, w)
\]
where $\tilde a(0, 0) > 0$. By Lemma 12.1 of \cite{DP1} it follows that $(Q_0 \sm \{0\}) \cap M \subset M^+_s$, i.e., $Q_0$ intersects $T^+_2$ only at the origin. Let $w_1 = 
u_1 + i v_1, w_2 = u_2 + i v_2$ and define
\begin{align*}
\ti {\cal L} &= \big\{ (z, w) \in U_2 \times U^{\ast}_2 : r(z, \ov w) = 0, u_1 = u_2 = 0 \big\}\\
             &= M^{\mbb C} \cap \big\{ (z, w) \in U_2 \times U^{\ast}_2 : u_1 = u_2 = 0 \big\}.
\end{align*}
Let $\pi : U_2 \times U^{\ast}_2 \ra U_2$ be the coordinate projection onto the $(z_1, z_2)$ variables. Then it can be seen that $\pi(\ti {\cal L}) = \cal L$. The 
equations that define $\ti {\cal L} \subset U_2 \times U^{\ast}_2$ are
\begin{align*}
f_1 &= x_2 + u_2 + \Re \big( (z_1 + \ov w_1)^{2m} \ti a(z_1, w) \big),\\
f_2 &= y_2 - v_2 + \Im \big( (z_1 + \ov w_1)^{2m} \ti a(z_1, w) \big),\\
f_3 &= u_1,\\
f_4 &= u_2
\end{align*}
where these are regarded as functions of $x_1, y_1, x_2, y_2, u_1, v_1, u_2$ and $v_2$. Define the map
\[
F : U_2 \times U^{\ast}_2 \ra \mbb R^4
\]
by $F = (f_1, f_2, f_3, f_4)$ and note that $\ti {\cal L} = F^{-1}(0)$. Then the derivative of $F$ at the origin is
\[
DF(0) =
 \begin{pmatrix}
 0 & 0 & 1 & 0 & 0 & 0 & 1 & 0\\
 0 & 0 & 0 & 1 & 0 & 0 & 0 & -1\\
 0 & 0 & 0 & 0 & 1 & 0 & 0 & 0\\
 0 & 0 & 0 & 0 & 0 & 0 & 1 & 0
 \end{pmatrix},
\]
where the rows are the gradients of $f_1, f_2, f_3, f_4$ with respect to the variables mentioned above in that order. This matrix has full rank, for the minor formed 
by the partial derivatives with respect to $x_2, y_2, u_1, u_2$ is non-zero. The implicit function theorem therefore shows that $\ti {\cal L}$ is a smooth real four dimensional 
manifold near the origin and the local coordinates on $\ti {\cal L}$ are given by $x_1, y_1, v_1, v_2$. Moreover, there are real analytic functions $h_1, h_2, h_3, h_4$ defined 
in a neighbourhood of the origin in the $x_1, y_1, v_1, v_2$ variables such that $\ti {\cal L}$ is described by
\begin{align*}
x_2 &= h_1(x_1, y_1, v_1, v_2),\\ 
y_2 &= h_2(x_1, y_1, v_1, v_2),\\
u_1 &= h_3(x_1, y_1, v_1, v_2),\\
u_2 &= h_4(x_1, y_1, v_1, v_2).
\end{align*}
Working with the $f_i$'s it can be seen that the real tangent space to $\ti {\cal L}$ at the origin is the direct sum of the complex line spanned by $z_1 = x_1 + i y_1$ and a 
totally real plane spanned by $v_1, v_2$. This description of the tangent space to $\ti {\cal L}$ persists in a neighbourhood of the origin, and therefore the CR dimension of 
$\ti {\cal L}$ is one near the origin.

\medskip

Now if $(0, w^0) \in \pi^{-1} \cap \ti {\cal L}$ then $0 \in Q_{w^0}$ and hence $w^0 \in Q_0$. But it is known that $Q_0$ intersects $T^+_2$ only at the origin, if $U_2$ is small 
enough, and therefore with this choice of $U_2$, it follows that $w^0 = 0$. This shows that $\pi : \ti {\cal L} \ra U_2$ is proper. Hence $\cal L$ is subanalytic and therefore 
admits (cf. \cite{Hi}) a subanalytic stratification by real analytic submanifolds. To see what the dimension of $\cal L$ is, observe that the holomorphic projection $\pi$ 
restricted to $\ti {\cal L}$ is of the form
\[
\pi(x_1, y_1, v_1, v_2) = (x_1, y_1, x_2, y_2)
\]
in terms of the local coordinates on $\ti {\cal L}$. The differential
\[
d \pi = 
 \begin{pmatrix}
 e_1\\
 e_2\\
 \nabla h_1\\
 \nabla h_2 
 \end{pmatrix}
\]
where $e_1 = (1, 0, 0, 0)$, $e_2 = (0, 1, 0, 0)$ and $\nabla h_1$, $\nabla h_2$ are the gradients with respect to $x_1, y_1, v_1, v_2$. To compute them, note that if $h = (h_1, 
h_2, h_3, h_4)$, the implicit function theorem again gives
\[
Dh = - \big(\pa f_i / \pa a_j \big)^{-1}_{i, j} \cdot \big( \pa f_i / \pa b_j)_{i, j},
\]
where $(a_1, a_2, a_3, a_4) = (x_2, y_2, u_1, u_2)$ and $(b_1, b_2, b_3, b_4) = (x_1, y_1, v_1, v_2)$. Therefore,
\[
Dh(0) = 
 \begin{pmatrix}
 0 & 0 & 0 & 0\\
 0 & 0 & 0 & 1\\
 0 & 0 & 0 & 0\\
 0 & 0 & 0 & 0
 \end{pmatrix},
\]
and hence $\nabla h_1(0) = (0, 0, 0, 0)$ and $\nabla h_2(0) = (0, 0, 0, 1)$. This implies that the rank of $d \pi$ equals $3$ at the origin. Note that $d \pi$ cannot have full 
rank, i.e., $4$ at points close to the origin, for if it does have full rank at $a \in \ti {\cal L}$ then
\[
d \pi(a) : T_a \ti {\cal L} \ra \mbb C^2
\]
is an isomorphism. But $T_a \ti {\cal L}$ is the sum of a complex line (close to $z_1 = x_1 + i y _1$) and a totally real subspace (close to that spanned by $v_1, v_2$). Hence 
the kernel of $d \pi(a)$ is the sum of a complex line (close to $z_2 = x_2 + i y_2$) and a totally real subspace (close to that spanned by $u_1, u_2$). This, however, is a 
contradiction since $\ker d \pi(a)$ must be a complex subspace. Thus the rank of $\pi$ equals $3$ everywhere on $\ti {\cal L}$ and the rank theorem combined with the properness 
of $\pi$ (cf. Proposition~3.5 and in particular Lemma~3.5.1 in \cite{Hi}) show that $\cal L = \pi(\ti {\cal L})$ is locally everywhere a finite union of real analytic three 
dimensional submanifolds of $U_2$.
\end{proof}

\begin{lem}
$\ov X_f$ is complex analytic near $\cal L \times T'^-_2$.
\end{lem}

\begin{proof}
It suffices to consider the behaviour of $X_f$ near $(a, a') \in (\ov X_f \sm X_f) \cap (\cal L \times T'^-_2)$. Fix small neighbourhoods $U_a, U'_{a'}$ around $a, a'$ respectively. We may assume that
\[
\cal L \cap U_a = \cal L_1 \cup \cal L_2 \cup \cdots \cup \cal L_{\mu},
\]
where each $\cal L_j$ is a closed, real analytic three dimensional submanifold of $U_a$. To start with, we will assume that $a \in \cal L_j \sm \cup_{i \not= j} \cal L_i$ 
for some $1 \le j \le \mu$, so that $a$ is a smooth point on $\cal L$. Since the CR dimension of $\cal L \times T'^-_2$ is one and $X_f$ has pure dimension two, it follows 
that $X_f$ admits analytic continuation, say $X^{ext}_f \subset U_a \times U'_{a'}$, which is a closed analytic set after shrinking these neighbourhoods if necessary. As 
before, let $\pi, \pi'$ be the coordinate projections onto the factors $U_a, U'_{a'}$ respectively, and define
\[
S = \big\{ (w, w') \in X^{ext}_f : \dim_{(w, w')} (\pi')^{-1}(w') \ge 1 \big\}.
\]
The defining condition for $X_f$, i.e., (2.1), forces $\pi' : X^{ext}_f \ra U'_{a'}$ to be locally proper, and hence $\pi'(X^{ext}_f)$ contains an open subset of $U'_{a'}$. 
Therefore, the Cartan-Remmert theorem (in \cite{L}, for example) implies that $\dim S \le 1$.

\medskip

Suppose that $(b, b') \in (\ov X_f \sm X_f) \sm S \subset \cal L \times T'^-_2$. It is then possible to choose neighbourhoods $W_b, W'_{b'}$ around $b, b'$ respectively 
such that the projection
\[
\pi': X^{ext}_f \cap (W_b \times W'_{b'}) \ra W'_{b'}
\]
is proper. Since $\ov X_f \sm X_f \subset \cal L \times T'^-_2$, it follows that $\pi'(X_f \cap (W_b \times W'_{b'}))$ contains a one-sided neighbourhood of $b'$, say, 
$\Om' \subset U'^+_1$. Evidently, $\pa \Om'$ contains a point from $M'^+_s$ and this contradicts the fact that the limit points of $X_f$ in $U^+_1 \times (U'\cap M')$ are 
contained in $\cal L \times T'^-_2$. Thus $\ov X_f \sm X_f \subset S$. But the three dimensional Hausdorff measure of $S$ is zero and Shiffman's theorem implies that $\ov 
X_f$ itself is analytic in $W_b \times W'_{b'}$. Therefore, $X^{ext}_f = \ov X_f$. This argument works when $a \in \cal L$ is a smooth point. If $a \in \cal L_{\alpha} \cap 
\cal L _{\beta}$ for $\alpha \not= \beta$, $1 \le \alpha, \beta \le \mu$, the theorems of Cartan-Bruhat (see for example \cite{Na}) show that the singular locus of 
$\cal L_{\alpha} \cap \cal L_{\beta}$ is contained in a real analytic set of strictly lower dimension. Thus it is possible to proceed by downward induction to conclude 
that $\cal L \times T'^-_2$ is a removable singularity for $X_f$.
\end{proof}

\begin{lem}
There exists a closed complex analytic set $\hat X_f \subset U_1 \times U'$ of pure dimension two such that $X_f \subset \hat X_f \cap (U^+_1 \times U'^+)$. In particular, $f$ 
extends holomorphically across $p \in T^+_2$.
\end{lem}

\begin{proof}
It suffices to show that $X_f$ can be continued across $\cal L \times M'_e$. For this, note that by Lemmas~5.1 and~5.4, the projection
\[
\pi: \ov X_f \sm (\cal L \times M'_e) \ra U^+_1
\]
is proper. The exceptional set $M'_e$ being a locally finite union of real analytic arcs and points is locally pluripolar and hence globally so by Josefson's theorem. Let 
$\varrho$ be a plurisubharmonic function on $\mbb C^4$ such that $U^+_1 \times M'_e \subset \{ \varrho = -\infty \}$. Since $M$ is of finite type near $p$ it follows that 
$\cal L \cap M$ has real dimension at most two and hence it is possible to choose $p^j$, across which $f$ holomorphically extends, to not lie on $\cal L \cap M$. Fix a 
small ball $B \subset U^+_1 \sm \cal L$ on which $f$ is well defined and note that by Lemma 5.2, $\ov X_f$ has no limit points on $B \times (U'\cap M')$. Therefore, the 
pluripolar set $\{ \varrho = -\infty \}$ is a removable singularity for the non-empty analytic set $(\ov X_f \sm (\cal L \times M'_e)) \sm \{ \varrho = -\infty \}$ by 
Bishop's theorem. Hence $\ov X_f \subset U^+_1 \times U'$ is analytic and the projection 
\[
\pi : \ov X_f \ra U^+_1
\]
remains proper. The coordinate functions $z'_i$ (for $i = 1, 2$) restricted to $\ov X_f$ satisfy a monic polynomial whose coefficients are holomorphic functions on 
$U^+_1$. By Trepreau's theorem  \cite{Tr}, each of these functions extend to a fixed, full neighbourhood of $p$ in $\mbb C^2$ and the zero locus of the resulting pair of polynomials, 
which are still monic in $z'_i$ (for $i = 1, 2$) provides the continuation of $X_f$ as an analytic set, say $\hat X_f \subset U_1 \times U'$. By Theorem 7.4 of \cite{DP1} 
it follows that $f$ extends holomorphically across $p \in T^+_2$.
\end{proof}

\subsection{The case when $\mathbf{p \in T^+_1 \cup T^+_0}$} Let $\gamma \subset T^+_1$ be a smooth real analytic arc and suppose that $p \in \gamma$. Then
\[
C = \{ w \in U_1 : \gamma \cap U_1 \subset Q_w \}
\]
is a finite set. Indeed, $\gamma \cap U_1 \subset Q_w$ implies that $Q_w$ is the unique complexification of $\gamma \cap U_1$. Since the Segre map has finite fibres near 
$p$, it follows that $C$ must be finite. We may therefore assume that $Q_p \cap \gamma = \{p\}$ locally. All the previous arguments used in Subsection 5.1 can 
now be applied to show that $f$ holomorphically extends across $p \in \gamma$. The remaining set is discrete in $\gamma$, and again the same arguments apply to show the 
extendability of $f$ across all points on $T^+_1$ and hence also across $T^+_0$.


\section{Proof of Theorem 1.1 -- Case ($iii$)}

\no As discussed in Section 2, it suffices to consider the case when $p$ is on either a one or zero dimensional stratum of the border. Let $\gamma$ be a smooth real analytic 
arc in the border and suppose that $p \in \gamma$. We may also assume that $Q_p \cap \gamma = \{p\}$ locally near $p$ to start with. In particular, if $U$ is a neighbourhood of 
$p$ in $\mbb C^2$ it follows that $Q_p \cap M \cap \pa U$ is contained in the union of $M^+, M^-$ and the two dimensional strata of the border. Therefore, the behaviour of $f$ 
near points on $Q_p \cap M \cap \pa U$ is known by cases (i) and (ii) of the main theorem. Suppose that $p \in cl_f(p)\cap M'^+_s$. Choose a standard pair of neighbourhoods 
$U_1 \subset U_2$ 
around $p$ and $U'_1 \subset U'_2$ around $p'$ so that $X_f$ as in (2.1) is a non-empty closed complex analytic set of pure dimension two. Let $p^j$ be a sequence on $M$ 
converging to $p$, across which $f$ holomorphically extends for each $j$, and such that $f(p^j) \ra p'$ -- such a sequence exists by the reasoning given earlier. By shrinking 
$U'_1, U'_2$ if needed, we may assume that $Q'_{w'} \cap D'$ is relatively compactly contained in $U'_2$ and connected for all $w'\in U'^+_1$. Consider
\[
\tilde X_f = \big\{ (w, w') \in U^+_1 \times U'^+_1 : f({}_{{}^sw}Q_w) \subset Q'_{w'} \cap D \big\},
\]
which evidently contains $X_f$ near points of extendability of $f$ and is also a pure two dimensional local analytic set by the arguments of Lemma 4.1. It is also closed since 
$Q'_{w'} \cap D'$ is connected and compactly contained in $U'_2$, i.e., the germs $f({}_{{}^sw}Q_w)$ cannot escape $U'_2$ because they are contained in $Q'_{w'} \cap D'$. By 
analytic continuation, $X_f \subset \tilde X_f$ everywhere in $U^+_1 \times U'^+_1$. Now Proposition~4.3 in \cite{SV} shows that $X_f$ has no limit points either on $(U_1 \cap M) 
\times U'^+_1$ or $U^+_1 \times (U'\cap M')$. Let $X_j$ and $C_j$ be defined as in (4.3) and (4.4). The absence of limit points of $X_f$ on the aforementioned sets implies that 
$C_j$ is a one dimensional analytic set in $U_1 \times U'^+_1$. Again, Proposition~4.3 and Lemma~5.1 in \cite{SV}
(see the proof of Lemma~4.8 in Section~4 as well), show that 
$\ov C_j$ is analytic in $U_1 \times U'_1$. By Lemma 5.2 of \cite{SV} (or by Lemma 4.8 in Section~4) it follows that the cluster set of $\{C_j\}$ is non-empty in $U^+_1 \times 
U'^+_1$. Moreover, the volumes of $\{C_j\}$ are locally uniformly bounded in $U^+_1 \times U'^+_1$, and hence the sequence converges to a pure one dimensional analytic set, say 
$C_p \subset U^+_1 \times U'^+_1$, that contains $(p, p')$ in its closure. Furthermore, by continuity it is evident that $C_p \subset Q_p \times Q'_{p'}$. Thus there are 
points on $Q_p \sm \{p\}$ over which $X_f$ is a well defined ramified cover. By adapting the arguments in \cite{Sh} as done in case (ii) of Lemma 4.6, it follows that the 
graph of $f$ extends as an analytic set near $(p, p')$, and hence $f$ extends holomorphically across $p$ with $f(p) = p'$. This is clearly a contradiction for it is 
possible to find strongly pseudoconcave points near $p$ that are mapped locally biholomorphically to strongly pseudoconvex points near $p'$. It may be noted that the 
properness of $f$ is not required here -- it suffices for $f$ to have discrete fibres near $p$ and this is guaranteed by Proposition~3.2. Thus the following lemma has 
been proved:

\begin{lem}
Let $D, D'$ be domains in $\mbb C^2$, both possibly unbounded and $f : D \ra D'$ a holomorphic mapping. Suppose that $M \subset \pa D$ and $M'\subset \pa D'$ are open 
pieces which are smooth real analytic and of finite type and let $p \in \gamma$ where $\gamma$ is a one dimensional stratum in the border between the pseudoconvex and 
pseudoconcave points on $M$. Let $U$ be an open neighbourhood of $p$ in $\mbb C^2$ such that the cluster set of $U \cap M$ is bounded and contained in $M'$. Then the cluster 
set of $p$ does not intersect $M'^+_s$.
\end{lem}

Continuing with the proof of Theorem 1.1 (iii), note that since $f$ is proper, it follows that $cl_f(p)$ cannot contain points on $M'$ that lie in $\hat D'$. Indeed, if $p'\in 
cl_f(p) \cap M' \cap \hat D'$, the cluster set of $p'$ under the correspondence $f^{-1} : D'\ra D$ will only contain points in $\hat D$ by Lemma 3.1 of \cite{DP1} and in 
particular it would follow that $p \in \hat D$. Hence $f$ would extend across $p$. Now choose a standard pair of neighbourhoods $U_1 \subset U_2$ around $p$ and a 
neighbourhood $U'$ containing $cl_f(U_2 \cap M)$ as done in Section~5, so that $X_f$ (as in (2.1)) is a non-empty closed complex analytic set in $U^+_1 \times U'^+$. By Lemma 5.1, it follows that $X_f$ has no limit points on $U^+_1 \times (\pa U'\sm D')$. Define
\[
\cal L_{\gamma} = \bigcup_{w \in \gamma \cap U_1} Q_w,
\]
which is seen to be locally foliated by open pieces of Segre varieties at all its regular points and have CR dimension one by the reasoning given in Lemma 5.3. Then by Lemmas~5.4 and~6.1, $\ov X_f$ is analytic near $\cal L_{\gamma} \times T'^+_2$. Again, by Lemma~6.1 and the remark made above about $p$ not belonging to $\hat D$, 
it follows that $X_f$ possibly has limit points only on $U^+_1 \times (M'_e \cup T'^+_1 \cup T'^+_0)$, which is a pluripolar set. The arguments used in Lemma~5.5 show that there is a closed 
complex analytic set $\hat X_f \subset U_1\times U'$ that extends the graph of $f$ near $\{p\} \times cl_f(p)$, and hence $f$ extends holomorphically across $p$. The remaining 
set $C$ is discrete in $\gamma$ and the same arguments apply to show that $f$ extends across each point on $\gamma$ and hence across the zero dimensional strata on the border 
as well. This completes the proof of Theorem 1.1 (iii).


\end{document}